\documentclass[a4paper]{article}

\usepackage[latin1]{inputenc} 
\usepackage[T1]{fontenc} 
\usepackage[left=2.2cm,right=2.2cm,top=2cm,bottom=2.2cm]{geometry}
\usepackage{amsmath,amsthm,amssymb} 
\usepackage[dvipsnames,table]{xcolor} 
\usepackage{graphicx}
\usepackage{mathtools} 
\usepackage{stmaryrd} \SetSymbolFont{stmry}{bold}{U}{stmry}{m}{n} 
\usepackage{esint} 
\usepackage[english]{babel} 
\usepackage{float} 
\usepackage{enumitem}
\usepackage{multirow,multicol,paracol,booktabs} 
\usepackage{lipsum} 
\usepackage[normalem]{ulem} 
\usepackage{tabularx}\newcolumntype{C}{>{\centering\arraybackslash}X} 
\usepackage{xspace}
\usepackage{clipboard} 
\usepackage[overload]{empheq} 
\usepackage{mathrsfs} 
\usepackage[backend=bibtex,giveninits=true,style=trad-alpha,url=false,doi=false,isbn=false,sorting=nty,style=alphabetic,maxnames=50]{biblatex}
\usepackage[hidelinks,colorlinks,allcolors=.,urlcolor=teal]{hyperref} 
\usepackage[noabbrev]{cleveref} 

\crefdefaultlabelformat{#2#1#3}
\crefname{equation}{}{} \creflabelformat{equation}{(#2#1#3)}
\crefname{section}{section}{sections} \Crefname{section}{Section}{Sections}
\crefname{subsection}{}{} \Crefname{subsection}{Section}{Sections}
\crefname{enumi}{}{}
\creflabelformat{appx}{#2Appendix#3}\Crefname{appx}{}{}

\allowdisplaybreaks 

\AtEveryBibitem{\clearfield{location}}
\AtEveryBibitem{\clearfield{eprint}}
\AtEveryBibitem{\clearname{editor}}
\AtEndDocument{
	\printbibliography
}

\setlist[itemize,1]{label=$-$}
\setlist[itemize,2]{label=$-$}
\setlist[itemize,3]{label=$-$}

\makeatletter
\NewCommandCopy\latexparagraph\paragraph
\RenewDocumentCommand{\paragraph}{sO{#3}m}{%
  \IfBooleanTF{#1}
    {\latexparagraph*{\maybe@addperiod{#3}}}
    {\latexparagraph[#2]{\maybe@addperiod{#3}}}%
}
\newcommand{\maybe@addperiod}[1]{%
  #1\@addpunct{.}%
}
\makeatother

\makeatletter
\newcommand{\ncite}{\@ifstar{\@ncite}{\@@ncite}}
\newcommand{\@ncite}[2][]{\citeauthor*{#2} \ifx#1\empty\cite{#2}\else\cite[#1]{#2}\fi}
\newcommand{\@@ncite}[2][]{\citeauthor{#2} \ifx#1\empty\cite{#2}\else\cite[#1]{#2}\fi}
\makeatother

\newtheorem{theorem}{Theorem}[section] \crefname{theorem}{theorem}{Theorem} \Crefname{theorem}{Theorem}{Theorems}
\newtheorem*{theorem*}{Theorem}

\newtheorem{lem}[theorem]{Lemma} \crefname{lem}{lemma}{lemmata} \Crefname{lem}{Lemma}{Lemmata}
\newtheorem*{lem*}{Lemma}

\newtheorem{prop}[theorem]{Proposition} \crefname{prop}{proposition}{propositions} \Crefname{prop}{Proposition}{Propositions}
\newtheorem*{prop*}{Proposition}

\newtheorem{cor}[theorem]{Corollary} \crefname{cor}{corollary}{corollaries} \Crefname{cor}{Corollary}{Corollaries}
\newtheorem*{cor*}{Corollary}

\theoremstyle{definition}

\newtheorem{definition}[theorem]{Definition} \crefname{definition}{definition}{definitions} \Crefname{definition}{Definition}{Definitions}
\newtheorem*{definition*}{Definition}

 \crefname{conjecture}{conjecture}{conjectures} \Crefname{conjecture}{Conjecture}{Conjectures}
\newtheorem*{conjecture*}{Conjecture}

\newtheoremstyle{theassumptionstyle}{\topsep}{\topsep}{}{}{}{}{.5em}{\textbf{\thmname{#1} [A\thmnumber{#2}]}\thmnote{ (#3)}.}
\theoremstyle{theassumptionstyle}  
\newtheorem{hyp}[theorem]{Assumption} \crefname{hyp}{}{} \Crefname{hyp}{}{} \creflabelformat{hyp}{\textbf{[#2A#1#3]}}

\theoremstyle{remark}

 \crefname{example}{example}{examples} \crefname{example}{Example}{Examples}

\newtheorem{rem}[theorem]{Remark} \crefname{rem}{remark}{remarks} \Crefname{rem}{Remark}{Remarks} 
\newtheorem*{rem*}{Remark}

\newcommand{\mysubcaption}[2][0.8]{\vspace*{5pt}\begin{minipage}{#1\linewidth}\begin{center}\footnotesize\emph{#2}\end{center}\end{minipage}}
\newcommand{\highlight}[2]{\begin{center}\small\begin{minipage}{0.9\textwidth}\begin{center}\textbf{#1}\end{center}\vspace*{-7pt}#2\end{minipage}\end{center}}

\DeclareMathOperator*{\dom}{dom\,}

\DeclareMathOperator*{\supp}{supp\,}
\DeclareMathOperator*{\diam}{diam\,}
\DeclareMathOperator*{\conv}{conv\,}
\DeclareMathOperator*{\vect}{vect\,}

\DeclareMathOperator*{\graph}{Graph}

\newcommand{\ind}{{\textrm{1\hspace{-0.6ex}I}}}
\newcommand{\st}{\ \middle| \ } 
\newcommand{\bdot}{{\boldsymbol{\cdot}}} 

\newcommand{\resmes}{\mathbin{\vrule height 1.6ex depth 0pt width 0.18ex\vrule height 0.15ex depth 0pt width 0.7ex}} 

\makeatletter
\newcommand{\stephighlight}[1]{\medskip\noindent\textbf{#1\@addpunct{.}}\quad} 
\makeatother 

\newcommand{\Sp}{\Omega}
\DeclareMathOperator{\T}{T} 
\DeclareMathOperator{\Psp}{{\mathscr{P}}} 
\DeclareMathOperator{\Tan}{\mathbf{Tan}} 
\DeclareMathOperator{\Sol}{\mathbf{Sol}} 
\DeclareMathOperator{\Set}{\mathbf{Set}} 
\DeclareMathOperator{\M}{{\mathscr{M}}} 

\DeclareMathOperator*{\vspan}{span\,} 

\newcommand{\bary}[2]{{\text{Bary}_{#2}\left(#1\right)}} 

\newcommand{\Zaitchek}{Zaj\'i\v{c}ek\xspace}
\newcommand{\DC}[1]{DC\textsubscript{$#1$}} 
\newcommand{\sDC}[1]{$\sigma-$\DC{$#1$}} 
\newcommand{\filler}[1]{#1} 

\title{Locality of centred tangent cones in the Wasserstein space}
\author{Averil Aussedat\footnote{Dipartimento di Matematica, Universit\`a di Pisa, largo Pontecorvo 5, 56127 Pisa, Italy. \texttt{averil.aussedat@dm.unipi.it}.}}
\date{}

\begin{document}

\maketitle

\highlight{Abstract}{The geometric tangent cone to a probability measure $\mu$ is a set of measure-valued applications that are almost geodesics. This is a nonlocal condition, typically lost when conditioning the measure on a given set. We show that if one removes the barycenter of any element of the tangent cone, then the resulting set of centred measure fields is characterized by a local condition. Precisely, centred tangent fields must be concentrated on a family of vector subspaces attached to any point, and these subspaces correspond to the normal spaces to some sets of ``dimension $k$'' on which the measure $\mu$ is concentrated.}

\medskip

\medskip


\section*{Introduction}

\renewcommand{\Sp}{\mathbb{R}^d}

In his inspired article \cite{lottTangentConesWasserstein2016}, Lott observed that whenever $\mu$ is the restriction of a Hausdorff measure to a $k-$dimensional $C^2$ manifold, the measures in the geometric tangent cone to $\mu$ have a particular structure. It was already known that the barycenter of such a measure must belong to the $L^2_{\mu}$ closure of gradients of smooth compactly supported functions, as proved in full generality in \cite{gigliInverseImplicationBrenierMcCann2011}. However, when removing the barycenter, the remaining \emph{centred part} happens to be concentrated on the normal directions to the manifold. The purpose of this work is to investigate the general case, using a new line of reasoning that allows us to remove any assumption on $\mu \in \Psp_2(\Sp)$.

\medskip 

Our main result is a combination of \Cref{res:decomp}, \Cref{res:charmuk} and \Cref{res:PeqD} below, gathered here in an single statement. A set $A \subset \mathbb{R}^d$ is \DC{k} if, up to a permutation of variables, it is the graph of a map from $\mathbb{R}^k$ to $\mathbb{R}^{d-k}$ with all coordinates being Differences of Convex (DC) functions. A set that can be covered by countably many \DC{k} sets is said to be \sDC{k}.
\begin{theorem*}
	Let $\mu \in \Psp_2(\mathbb{R}^d)$. There exists a unique decomposition $\mu = \sum_{k=0}^d m_k \mu^k$ in mutually singular measures such that $m_k \mu^k$ gives 0 mass to \DC{k-1} sets and is concentrated on a \sDC{k} set $A_k$. $\vphantom{\sum_{k=0}^d}$ 
	Moreover, the centred tangent cone $\Tan_{\mu}^0$ splits in $\sum_{k=0}^d m_k \Tan_{\mu^k}^0$, and a centred measure $\xi^k$ belongs to $\Tan_{\mu^k}^0$ if and only if it is concentrated on the $(d-k)-$dimensional normal spaces to $A_k$, $\vphantom{\sum_{k=0}^d}$which exist $\mu^k-$almost everywhere. 
\end{theorem*}

For instance, in dimension $d=2$, let $\mu = \frac{1}{3} \mu^0 + \frac{1}{3} \mu^1 + \frac{1}{3} \mu^2$, where $\mu^0$ is atomic, $\mu^1$ nonatomic and supported on $[0,1] \times \{0\}$, and $\mu^2$ is absolutely continuous. Let $\xi$ be a probability measure on $\T\Sp \simeq \mathbb{R}^d \times \mathbb{R}^d$ with finite second moment, first marginal $\mu$ and barycenter 0 in each fiber. Since the $\mu^k$ are mutually singular, there is a unique way to write $\xi = \frac{1}{3} \sum_{k=0}^2 \xi^k$ with $\pi_{x \#} \xi^k = \mu^k$ for each $k$. The theorem then states that $\xi$ is tangent to $\mu$ if and only if $\xi^1$ is concentrated on pairs $(x,v)$ with $v$ orthogonal to $[0,1] \times \{0\}$, and $\xi^2$ is concentrated on pairs $(x,v)$ with $v$ orthogonal to $\mathbb{R}^2$, hence $\xi^2 = (id,0)_{\#} \mu^2$. 

\medskip

In dimension one, the decomposition $\mu = m_0 \mu^0 + m_1 \mu^1$ reduces to the atomic/diffuse decomposition. In this case, the statement was already proved in \cite{sedjroCharacterizationTangentSpace2014} (see also \cite{aussedatStructureGeometricTangent2025}): it says that a centred field $\xi = m_0 \xi^0 + m_1 \xi^1$ is tangent if and only if $\xi^1 = (id,0)_{\#} \mu^1$. This case is quite particular, since \sDC{0} sets are exactly countable sets, and the only \sDC{1} set is $\mathbb{R}$ itself. In the general case, one has to care about the regularity of the sets $A_k$ for the intermediate dimensions $0 < k < d$. 

\medskip

The relation between \DC{d-1} sets and optimal transport plans for the 2-Wasserstein distance appeared in the work of \ncite{gigliInverseImplicationBrenierMcCann2011}, which characterizes the measures $\mu$ satisfying the conclusion of the Brenier-McCann theorem \cite{brenierPolarFactorizationMonotone1991}. Precisely, a measure $\mu$ that gives 0 mass to any \DC{d-1} set (and consequently, to any \DC{k} set for $k \leqslant d-1$) does not admit optimal plans that \emph{split mass}, in the sense that for any other measure $\nu$ and optimal transport plan $\alpha$ between $\mu$ and $\nu$, there exists a map $f \in L^2_{\mu}(\Sp;\Sp)$ such that $\alpha = (id,f)_{\#} \mu$. The argument runs as follows: optimal plans are known to be concentrated on $c$-cyclically monotone sets, which in the case of $c(x,y) = |x-y|^2$, coincide with subdifferentials of convex functions. If an optimal plan splits mass at some point $x$, then the convex functions in question must admit several elements in their subdifferential at $x$, hence be non-differentiable. A remarkable theorem of \Zaitchek \cite{zajicekDifferentiationConvexFunctions1979} shows that the set of non-differentiability of a convex function can be covered by countably many \DC{d-1} sets, and that any \DC{d-1} set is contained in the set of non-differentiability of some convex function. Hence $\mu$ allows to split mass if and only if it charges a \DC{d-1} set.

\medskip 

This work generalizes Gigli's theorem by identifying which part of $\mu$ allows optimal plans to split mass in exactly $d-k$ directions, again using \Zaitchek's theorem. Moreover, and this is the main point, each part is concentrated on a set $A_k$ which admits normal spaces of dimension $d-k$ at $\mu-$almost any point, and the centred tangent measure fields are exactly the measure fields concentrated on these normal spaces. The intuition behind this decomposition is that tangent measures try to get away from $\mu$ as fast as possible, therefore choosing normal directions to the concentration sets. 

It may seem surprising that no regularity assumption on $A_k$ is needed: this is still a consequence of \Zaitchek's theorem, since quite miraculously, singularities of convex functions are described using (differences of) convex functions. Indeed, the regularity of $A_k$ is precisely the one that allows a convex function to be not differentiable in $d-k$ independent directions. By \Zaitchek, the singular subset of $A_k$ (at which the normal directions are not defined) has the ``same size'' as non-differentiability sets of convex functions \emph{in at least one more direction}. In the decomposition, these sets are seen by the lower-order measures $\mu^{j}$ for $j < k$, so that each $A_k$ has a tangent plane $\mu^k-$almost everywhere.

Since many properties of $\Tan_{\mu}^0$ only depend on the fact that this set is convex in a certain sense, we first focus on a generic horizontally convex closed cone of centred fields. However, the tangent cone is constructed in a canonical way for any $\mu$, and can be compared for different measures. As a side corollary, we show that the metric orthogonal $\Sol_{\mu}$ of $\Tan_{\mu}$ is closed with respect to the Wasserstein distance over $\T\Sp$.

\medskip

We mention that the statement is only concerned about centred fields. This is a natural assumption, since the opposite case of barycentric fields (induced by a map) cannot be described by a local condition. However, the following can be deduced for any measure field $\xi \in \Psp_2(\T\Sp)_{\mu}$. First, one can always write $\xi = (\pi_x, \pi_v + b(\pi_x))_{\#} \zeta$, where $b \in L^2_{\mu}$ is the barycenter map, and $\zeta$ is centred. It is known that $\xi \in \Tan_{\mu}$ if and only if both $(id,b)_{\#} \mu$ and $\zeta$ are tangent measure fields: hence the theorem applies to $\zeta$. Moreover, anticipating the notations of \Cref{sec:tansol}, we can further decompose $b = b_D + b_{D^{\perp}}$, where $b_D(x)$ is the projection of $b(x)$ on the vector space $D(x)$, and $b_{D^{\perp}} = b - b_D$. \Cref{res:charTanSol0} implies that the measure field $\frac{1}{2} \left[(id,-b_{D^{\perp}})_{\#} \mu + (id,b_{D^{\perp}})_{\#} \mu\right]$ is tangent, regardless of the fact that $b$ is or not. From this, one recovers that $(id,b_{D^{\perp}})_{\#} \mu$ is itself tangent, once again regardless of $b$ (see \cite[Proposition~5.4.4]{aussedatOptimalControlProblems2025}). Hence, the fact that $(id,b)_{\#} \mu$ is tangent only imposes (nonlocal) restrictions on the component $b_D$. This extends the observation of \cite{lottTangentConesWasserstein2016} that the barycenter of a tangent map is only constrained through its tangential component. 

\medskip

This article is organized as follows; \Cref{sec:cones} is devoted to closed convex cones of centred measure fields. \Cref{sec:tansol} applies the previous results to the centred geometric tangent cone $\Tan_{\mu}^0$ and its metric orthogonal $\Sol_{\mu}^0$, providing the candidate decomposition $\mu = \sum m_k \mu^k$. The fact that each $\mu^k$ is concentrated on a \sDC{k} set, and gives 0 mass to \DC{k-1} sets, is proved in \Cref{sec:concentration}. \Cref{sec:tangency} shows that the normal directions are well-defined $\mu^k-$almost everywhere, and characterize centred tangent measures. In addition, it is showed that the tangent measures in the sense of Preiss are supported on planes associated to $\Sol_{\mu}^0$ at $\mu-$almost every point. An \Cref{sec:appendix} collects some lengthy proofs. 

\renewcommand{\Sp}{\Omega}

\tableofcontents

\section{Preliminaries}


In the sequel, $\Sp = \mathbb{R}^d$ for $d \geqslant 1$. We keep the notation $\Sp$ to distinguish points $x \in \Sp$ from vectors $v \in \mathbb{R}^d$. 
Open balls of center $x$ and radius $r \geqslant 0$ are denoted $\mathscr{B}(x,r)$. The closure of a set $A \subset \Sp$ is denoted $\overline{A}$, and its complement $A^c$. 

\paragraph{Wasserstein spaces}

The set of Borel probability measures on a Polish space $X$ is denoted $\Psp(X)$. If $X = X_1 \times X_2$ is a product space and $\mu \in \Psp(X)$, we write $\mu = \mu(dx_1,dx_2)$ to give names to the variables of $X$, to be used in the canonical projections $\pi_{x_i} : (x_1,x_2) \to x_i$ for $i \in \{1,2\}$. The set $\Psp_2(X) \subset \Psp(X)$ collects the measures with finite second moment, i.e. such that $\int_{x \in X} d^2(x,o) \, d\mu < \infty$ for some $o \in X$.
 
A measurable application $f : X \to Y$ between two Polish spaces $X,Y$ induces an application $f_{\#} : \Psp(X) \to \Psp(Y)$ by $(f_{\#} \mu)(A) = \mu(f^{-1}(A))$ for any measurable $A \subset Y$. We refer to $f_{\#} \mu$ as the \emph{pushforward} of $\mu$ by $f$. 

Given $\mu \in \Psp_2(X)$ and $\nu \in \Psp_2(Y)$, the set of \emph{transport plans} between $\mu$ and $\nu$ is defined as
\begin{align*}
	\Gamma(\mu,\nu) \coloneqq \left\{ \alpha = \alpha(dx,dy) \in \Psp(X \times Y) \st \pi_{x \#} \alpha = \mu \text{ and } \pi_{y \#} \alpha = \nu \right\}.
\end{align*}
The $2-$Wasserstein distance $W(\mu,\nu)$ between $\mu$ and $\nu$ is given by 
\begin{align*}
	W^2(\mu,\nu) \coloneqq \inf_{\alpha \in \Gamma(\mu,\nu)} \int_{(x,y) \in X \times Y} d^2(x,y) d\alpha(x,y).
\end{align*}
We refer the reader to \cite{santambrogioOptimalTransportApplied2015} for an introduction to this distance and optimal transport, only mentioning that the infimum is reached on a set of \emph{optimal transport plans} denoted $\Gamma_o(\mu,\nu)$. 

\paragraph{Measure fields}

Let $\T\Sp \coloneqq \left\{ (x,v) \st x \in \Sp, \ v \in \T_x\Sp \right\}$ be the tangent bundle of $\Sp$, isometric to $\Sp \times \mathbb{R}^d$. When useful, we also denote $\T^n\Sp \coloneqq \left\{ (x,v_1,v_2,\cdots,v_n) \st x \in \Sp, \ v_i \in \T_x\Sp \text{ for } i \in \llbracket1,n\rrbracket \right\}$. 

For any measure $\mu \in \Psp_2(\Sp)$, denote $\Psp_2(\T\Sp)_{\mu}$ the set of Borel probability measures $\xi = \xi(dx,dv)$ on $\T\Sp$, with finite second moment, and satisfying the marginal condition $\pi_{x \#} \xi = \mu$. These elements can be seen as measure-valued applications in $L^2_{\mu}$, generalizing vector fields; for this reason, we refer to them as \emph{measure fields}. The measure fields of the form $\xi = (\pi_x, \pi_y - \pi_x)_{\#} \eta$ for $\eta \in \Gamma_o(\mu,\nu)$ and $\nu \in \Psp_2(\Sp)$ are called \emph{velocities of geodesics}. 
Following \cite{gigliGeometrySpaceProbability2008}, we introduce a metric structure on $\Psp_2(\T\Sp)_{\mu}$ that takes into account the common marginal $\mu$. Namely, given $\xi,\zeta \in \Psp_2(\T\Sp)_{\mu}$, define
\begin{align*}
	\Gamma_{\mu}(\xi,\zeta) \coloneqq \left\{ \alpha = \alpha(dx,dv,dw) \in \Psp_2(\T^2\Sp) \st (\pi_x,\pi_{v})_{\#} \alpha = \xi \text{ and } (\pi_x,\pi_w)_{\#} \alpha = \zeta \right\}.
\end{align*}
These plans only move mass between pairs $(x,v)$ and $(y,w)$ such that $x = y$. They define a distance $W_{\mu}$ by
\begin{align*}
	W_{\mu}^2(\xi,\zeta) \coloneqq \inf_{\alpha \in \Gamma_{\mu}(\xi,\zeta)} \int_{(x,v,w) \in \T^2\Sp} \left|v - w\right|^2 d\alpha(x,v,w).
\end{align*}
The distance to the zero measure field is shortened in $\|\xi\|_{\mu} = W_{\mu}(\xi,0_{\mu}) = \big(\int_{(x,v)} |v|^2 d\xi\big)^{1/2}$. The distance $W_{\mu}$ induces a \emph{metric scalar product} $\left<\cdot,\cdot\right>_{\mu} : (\Psp_2(\T\Sp)_{\mu})^2 \to \mathbb{R}$ by
\begin{align*}
	\left<\xi,\zeta\right>_{\mu} 
	\coloneqq \frac{1}{2} \left[\|\xi\|_{\mu}^2 + \|\zeta\|_{\mu}^2 - W_{\mu}^2(\xi,\zeta)\right]
	= \sup_{\alpha \in \Gamma_{\mu}(\xi,\zeta)} \int_{(x,v,w) \in \T^2\Sp} \left<v,w\right> d\alpha(x,v,w).
\end{align*}

\paragraph{Subsets of measure fields}

The \emph{barycenter} of a measure field $\xi$ is the unique element $b \in L^2_{\mu}(\Sp;\mathbb{R}^d)$ satisfying $\int \varphi(x,v) d\xi = \int \varphi(x,b(x)) d\mu$ for any quadratically growing $\varphi \in \mathcal{C}(\T\Sp;\mathbb{R})$ that is linear in~$v$. The measure fields with barycenter 0 are called \emph{centred}, and the set of centred measure fields is denoted $\Psp_2(\T\Sp)_{\mu}^0$. 

For $\lambda \in \mathbb{R}$ and $\xi \in \Psp_2(\T\Sp)_{\mu}$, define $\lambda \cdot \xi \coloneqq (\pi_x, \lambda \pi_v)_{\#} \xi$. We say that a subset $\Set_{\mu} \subset \Psp_2(\T\Sp)_{\mu}$ is 
\begin{itemize}
\item a (positive) cone if $\lambda \cdot \xi \in \Set_{\mu}$ whenever $\xi \in \Set_{\mu}$ and $\lambda \geqslant 0$;
\item (horizontally) convex if for any $\lambda \in [0,1]$, $\xi,\zeta \in \Set_{\mu}$ and $\alpha \in \Gamma_{\mu}(\xi,\zeta)$, the measure field given by $(\pi_x,(1-\lambda)\pi_v+\lambda\pi_w)_{\#} \alpha$ also belongs to $\Set_{\mu}$.
\end{itemize}
We often omit the adjective ``positive'' in the sequel.
Horizontal convexity is stronger than geodesic convexity, since any transport plan is allowed to produce interpolating curves. It should also be distinguished from convexity in the Banach sense of measures, which involves curves of the form $(1-\lambda) \xi + \lambda \zeta$. 

\paragraph{Grassmannian sections}

To reduce terminology, let us agree that a \emph{Grassmannian section} $D : \Sp \rightrightarrows \mathbb{R}^d$ is a measurable multivalued application such that $D(x)$ is a vector subspace, possibly reduced to $\{0\}$, for all $x \in \Sp$. Here measurability is understood in the classical sense, i.e. if for any open set $\mathcal{O} \subset \mathbb{R}^d$, the set $\left\{ x \in \Sp \st D(x) \cap \mathcal{O} \neq \emptyset \right\}$ is measurable. The \emph{graph} of $D$ is the set $\graph D \coloneqq \left\{ (x,v) \st v \in D(x) \right\}$, which is a measurable subset of $\T\Sp$ \cite[Corollary 2.2]{rockafellarMeasurableDependenceConvex1969}.
We record here the following lemma for later use. 

\begin{lem}\label{res:DWTSpClosed}
	Let $\mu \in \Psp_2(\Sp)$ and $D : \Sp \rightrightarrows \mathbb{R}^d$ be a Grassmannian section. 
	The set of $\xi \in \Psp_2(\T\Sp)_{\mu}$ such that $\xi(\graph D) = 1$ is closed with respect to the Wasserstein distance $W_{\T\Sp}$ on the tangent bundle.
\end{lem}

\begin{proof}
	Using a Castaing representation of $D$ \cite[Theorem 3]{rockafellarMeasurableDependenceConvex1969}, one can find $d$ measurable functions $u_1,\cdots,u_d : \Sp \to \mathbb{R}^d$ such that $D(x) = \vspan \{u_1(x),\cdots,u_d(x)\}$ for any $x \in \Sp$. Let $\varepsilon > 0$. By Lusin's theorem \cite[\nopp 2.3.5]{federerGeometricMeasureTheory1996}, there exists a measurable set $B_\varepsilon \subset \Sp$ such that $\mu(B_\varepsilon) \geqslant 1-\varepsilon$ and each function $u_j$ coincides with a continuous function on $B_{\varepsilon}$. Since $\mu(B_\varepsilon) = \sup_{C_\varepsilon \subset B_\varepsilon, C_\varepsilon \text{ closed}} \mu(C_\varepsilon)$, we can find $C_\varepsilon \subset B_\varepsilon$ closed such that $\mu(C_\varepsilon) \geqslant 1-2\varepsilon$. The set 
	\begin{align*}
		G_{\varepsilon} 
		\coloneqq \bigcup_{x \in C_\varepsilon} \{x\} \times D(x)
		= \graph D \cap C_{\varepsilon} \times \mathbb{R}^d
		= \bigcup_{x \in C_\varepsilon} \{x\} \times \vspan \{u_1(x),\cdots,u_j(x)\}
	\end{align*}
	is measurable and closed, respectively by the first and second equality. Moreover, $\xi_n(G_{\varepsilon}) = \mu(C_\varepsilon) \geqslant 1-2 \varepsilon$ for all $n \in \mathbb{N}$. Since $\zeta \mapsto \int \ind_{G_{\varepsilon}} d\zeta$ is upper semicontinuous with respect to $W_{\T\Sp}$ \cite[Lemma 5.1.7]{ambrosioGradientFlows2005}, 
	\begin{align*}
		\xi(\graph D)
		\geqslant \xi(G_{\varepsilon}) 
		\geqslant \limsup_{n \to \infty} \xi_n(G_{\varepsilon}) 
		\geqslant 1 - 2 \varepsilon.
	\end{align*}
	Passing to the limit in $\varepsilon \searrow 0$, we conclude that $\xi(\graph D) = 1$. 
\end{proof}

\section{Closed convex cones of centred measure fields}\label{sec:cones}

\filler{In all this section, we consider a set $\Set_{\mu}^0 \subset \Psp_2(\T\Sp)_{\mu}^0$ with the following properties.}

\begin{hyp}\label{hyp:Set0}
	The set $\Set_{\mu}^0$ is a $W_{\mu}-$closed, horizontally convex cone of centred measure fields.
\end{hyp}

\filler{Here, and as below, ``cone'' means ``positive cone''. The prime motivation is the study of $\Tan_{\mu}^0$ and $\Sol_{\mu}^0$ to come in \Cref{sec:tansol}, which justifies our notation.}

\begin{prop}[Local characterization]\label{res:charSet0}
	A set $\Set_{\mu}^0$ satisfies \Cref{hyp:Set0} if and only if there exists a Grassmannian section $D : \Sp \rightrightarrows \mathbb{R}^d$ such that 
	\begin{align*}
		\xi \in \Set_{\mu}^0 
		\qquad \iff \qquad
		\left[\xi \in \Psp_2(\T\Sp)_{\mu}^0 \quad \text{and} \quad \xi(\graph D) = 1\right].
	\end{align*}
\end{prop}

As an example, in dimension $d = 2$, the set of centred $\xi$ satisfying $\left<v,(1,0)\right> = 0$ almost everywhere is $W_{\mu}-$closed, horizontally convex, stable by multiplication by any scalar, and characterized by the constant application $D \equiv \vect \{(0,1)\}$.

\begin{rem}[Negative cone]
	A direct consequence of \Cref{res:charSet0} is that $\Set_{\mu}^0$ is stable by multiplication by any scalar, including negative ones. Such a stability has already been noticed in the case of the tangent cone in \cite{gigliGeometrySpaceProbability2008}, using \emph{ad hoc} arguments; in this case, the property even holds for map-induced fields. Here it is mandatory to consider centred fields; indeed, already in dimension $d = 1$, the set of measure fields concentrated on $(x,v)$ with $v \geqslant 0$ provides a closed and horizontally convex cone that cannot be characterized by concentration over a Grassmannian section. 
\end{rem}

\begin{rem}[Vertical convexity]
	A somehow surprising implication of \Cref{res:charSet0} is that $\Set_{\mu}^0$ is convex as a subset of the Banach space of measures, in the sense that $(1-\lambda) \xi_0 + \lambda \xi_1 \in \Set_{\mu}^0$ whenever $\xi_0,\xi_1 \in \Set_{\mu}^0$ and $\lambda \in [0,1]$. Indeed, the superposition $(1-\lambda) \xi_0 + \lambda \xi_1$ stays centred, and concentrated on the graph of $D$. 
\end{rem}

\filler{
The aim of this section is to prove \Cref{res:charSet0}. Our strategy is to represent $\Set_{\mu}^0$ as the set of centred measure fields that are orthogonal to a family of ``simple'' measure fields, taken as in \Cref{def:gammaf} below. These fields will be used to construct the application $D$. The key observation, which justifies our interest for centred fields, is the following.}

\begin{lem}[Centred is local]\label{res:centredislocal}
	Let $\xi,\zeta$ be measure fields, with $\xi$ centred. Then $\left<\xi,\zeta\right>_{\mu} \geqslant 0$, and equality happens if and only if any disintegrations $(\xi_x)_x$, $(\zeta_x)_x$ satisfy $\left<\xi_x,\zeta_x\right>_{\delta_x} = 0$ for $\mu-$almost every $x$.
\end{lem}

\begin{proof}
	Owing to \cite[Proposition 4.2]{gigliGeometrySpaceProbability2008}, the metric scalar product writes as
	\begin{align}\label{centredislocal:formula}
		\left<\xi,\zeta\right>_{\mu} = \int_{x \in \Sp} \left<\xi_x, \zeta_x\right>_{\delta_x} d\mu(x).
	\end{align}
	Fix $x \in \Sp$. Identifying $\xi_x,\zeta_x$ with measures on $\Psp_2(\T_x\Sp)_{\delta_x} \simeq \Psp_2(\mathbb{R}^d)$, we may consider the product plan $\alpha_x \coloneqq \xi_x \otimes \zeta_x$ in the definition of the metric scalar product to get
	\begin{align*}
		\left<\xi_x, \zeta_x\right>_{\delta_x}
		\geqslant \int_{v,w \in \mathbb{R}^d} \left<v,w\right> d\alpha_x
		= \int_{v \in \mathbb{R}^d} v d\xi_x(v) \int_{w \in \mathbb{R}^d} w d\zeta_x(w).
	\end{align*}
	As the first term of the right hand-side vanishes $\mu-$almost everywhere, the integrand in \Cref{centredislocal:formula} is nonnegative. Hence $\left<\xi,\zeta\right>_{\mu} \geqslant 0$, and $\left<\xi,\zeta\right>_{\mu} = 0$ implies $\left<\xi_x,\zeta_x\right>_{\delta_x} = 0$ for $\mu-$a.e. $x$. The converse is direct from \Cref{centredislocal:formula}.
\end{proof}

\filler{This observation is used as follows. In the sequel, one often has to prove that $\left<\xi,\zeta\right>_{\mu} = 0$ for some $\xi,\zeta$ lying in cones of interest. By specific arguments, one gets to a first inequality $\left<\xi,\zeta\right>_{\mu} \leqslant 0$. In the Hilbertian case, when working with two-sided cones, one would typically take $- 1 \cdot \zeta$ in place of $\zeta$, and conclude with the same inequality that $\left<\xi,\zeta\right>_{\mu} = 0$. However, we do not \emph{a priori} assume that $\Set_{\mu}^0$ is stable by multiplication with a negative scalar, and the missing inequality is provided by \Cref{res:centredislocal}.}

\subsection{The metric orthogonal complement}

Our aim here is to write $\Set_{\mu}^0$ as the orthogonal complement of its orthogonal complement. For convenience, we restrict to centred measure fields, since any measure field induced by a map is orthogonal (with respect to $\left<\cdot,\cdot\right>_{\mu}$) to any centred measure field. 

\begin{lem}\label{res:orthGamma}
	Let $\Set_{\mu}^0$ satisfy \Cref{hyp:Set0}. Then 
	\begin{align*}
		(\Set_{\mu}^0)^{\perp 0} \coloneqq \left\{ \zeta \in \Psp_2(\T\Sp)_{\mu}^0 \st \left<\xi,\zeta\right>_{\mu} = 0 \text{ for all } \xi \in \Set_{\mu}^0 \right\}
	\end{align*}
	also satisfies \Cref{hyp:Set0}, i.e. $(\Set_{\mu}^0)^{\perp 0}$ is a $W_{\mu}-$closed, horizontally convex cone of centred measure fields.
\end{lem}

\begin{proof}
	By construction, the scalar product is continuous with respect to $W_{\mu}$, so that $(\Set_{\mu}^0)^{\perp 0}$ is $W_{\mu}-$closed. Let $\zeta_0,\zeta_1 \in \Set_{\mu}^0$, $\alpha \in \Gamma_{\mu}(\zeta_0,\zeta_1)$ and $\lambda \in [0,1]$. The measure field $\zeta_{\lambda} \coloneqq (\pi_x,(1-\lambda) \pi_v + \lambda \pi_w)_{\#} \alpha$ is centred. To show that $\zeta_{\lambda} \in (\Set_{\mu}^0)^{\perp 0}$, let $\xi \in \Set_{\mu}^0$, and $\beta \in \Gamma_{\mu}(\zeta_{\lambda},\xi)$. We construct a plan $\omega = \omega(dx,dv_0,dv_1,dw) \in \Gamma_{\mu}(\zeta_0,\zeta_1,\xi)$ such that $(\pi_x,\pi_{v_0},\pi_{v_1})_{\#} \omega = \alpha$ and $(\pi_x,(1-\lambda) \pi_{v_0} + \lambda \pi_{v_1}, \pi_w)_{\#} \omega = \beta$ as follows: first change variables to consider $\widetilde{\alpha} \coloneqq (\pi_x,\pi_{v},(1-\lambda) \pi_{v} + \lambda \pi_w)_{\#} \alpha \in \Gamma_{\mu}(\zeta_0,\zeta_{\lambda})$, then glue $\widetilde{\alpha}$ and $\beta$ along their common marginal $\zeta_{\lambda}$ by \cite[Lemma 5.3.2]{ambrosioGradientFlows2005} to produce $\widetilde{\omega} = \widetilde{\omega}(dx,dv,dv_{\lambda},d w)$, and define $\omega \coloneqq (\pi_x,\pi_{v},\lambda^{-1}(\pi_{v_{\lambda}} - (1-\lambda) \pi_{v}), \pi_w)_{\#} \widetilde{\omega}$. Then 
	\begin{align*}
		\int_{(x,v,w)} \left<v,w\right> d\beta
		= \int_{(x,v_0,v_1,w)} \left<(1-\lambda) v_0 + \lambda v_1, w\right> d\omega
		\leqslant (1-\lambda) \left<\zeta_0,\xi\right>_{\mu} + \lambda \left<\zeta_1,\xi\right>_{\mu}
		= 0.
	\end{align*}
	Passing to the supremum over $\beta$, we get that $\left<\zeta_{\lambda}, \xi\right>_{\mu} \leqslant 0$. By \Cref{res:centredislocal}, the inequality $\left<\zeta_{\lambda}, \xi\right>_{\mu} \geqslant 0$ holds as both fields are centred. Hence $\zeta_{\lambda} \in (\Set_{\mu}^0)^{\perp 0}$.	
	Lastly, if $\zeta \in (\Set_{\mu}^0)^{\perp 0}$ and $\lambda \geqslant 0$, there holds $\left<\xi, \lambda \cdot \zeta\right>_{\mu} = \left<\lambda \cdot \xi, \zeta\right>_{\mu} = 0$ for all $\xi \in \Set_{\mu}^0$, so that $\lambda \cdot \xi \in (\Set_{\mu}^0)^{\perp 0}$. 
\end{proof}

\filler{The next result uses elementary tools from the pseudo-Hilbertian structure of $\Psp_2(\T\Sp)_{\mu}$, and is very close to similar statements in \cite{gigliGeometrySpaceProbability2008,aussedatStructureGeometricTangent2025}. For this reason, we delay the proof to the \Cref{sec:appendix}.}

\begin{prop}[Projection]\label{res:projConvCone}
	Let $C$ satisfy \Cref{hyp:Set0}. Then any $\xi \in \Psp_2(\T\Sp)_{\mu}^0$ admits a unique metric projection $\pi_C^{\mu} \xi \in C$, realizing $\inf_{\zeta \in C} W_{\mu}(\xi,\zeta)$. Moreover, for any $\alpha \in \Gamma_{\mu,o}(\xi,\pi_C^{\mu} \xi)$, the measure field $(\pi_x,\pi_v - \pi_w)_{\#} \alpha$ is the metric projection of $\xi$ on $C^{\perp} \coloneqq \left\{ \zeta \in \Psp_2(\T\Sp)_{\mu} \st \left<\zeta,\gamma\right>_{\mu} = 0 \ \ \forall \gamma \in C \right\}$, and there holds
	\begin{align}\label{projConvCone:statement}
		\left<\xi, \zeta\right>_{\mu} \leqslant \left<\pi_C^{\mu} \xi, \zeta\right>_{\mu}
		\qquad \forall \zeta \in C,
	\end{align}
	with equality if $\zeta = - 1 \cdot \zeta = (\pi_x, - \pi_v)_{\#} \zeta$. 
\end{prop}

\begin{rem}
	The conclusions of \Cref{res:projConvCone} are not sharp; as a consequence of \Cref{res:charSet0}, equality will hold in \Cref{projConvCone:statement} for any $\zeta \in C$. One could additionally prove that $\Gamma_{\mu,o}(\xi,\pi_C^{\mu} \xi)$ reduces to a singleton that is induced by a map in a certain sense, following the reasoning of \cite[Theorem 4.33]{gigliGeometrySpaceProbability2008}. 
\end{rem}

\begin{lem}\label{res:biorth}
	There holds $\Set_{\mu}^0 = ((\Set_{\mu}^0)^{\perp 0})^{\perp 0}$.
\end{lem}

\begin{proof}
	The inclusion $\Set_{\mu}^0 \subset ((\Set_{\mu}^0)^{\perp 0})^{\perp 0}$ holds by definition. Conversely, let $\xi \in ((\Set_{\mu}^0)^{\perp 0})^{\perp 0}$, and denote by $\pi^{\mu} \xi$ its metric projection on $\Set_{\mu}^0$. By \Cref{res:projConvCone}, for any $\alpha = \alpha(dx,dv,dw) \in \Gamma_{\mu,o}(\xi,\pi^{\mu} \xi)$, the measure field $(\pi_x,\pi_v - \pi_w)_{\#} \alpha$ is the projection of $\xi$ on $(\Set_{\mu}^0)^{\perp 0}$. However, this projection is $0_{\mu}$, since $W_{\mu}^2(\xi,\zeta) = \|\xi\|_{\mu}^2 + \|\zeta\|_{\mu}^2$ for any $\zeta \in (\Set_{\mu}^0)^{\perp 0}$. Hence $v = w$ for $\alpha-$a.e. $(x,v,w)$, and $\xi = \pi^{\mu} \xi \in \Set_{\mu}^0$. 
\end{proof}

\subsection{Reduction to symmetric measure fields}

It will be useful to introduce the following notation: to any $f \in L^2_{\mu}(\Sp;\mathbb{R}^d)$, associate the measure field
\begin{align}\label{def:gammaf}
	\gamma_{f} \coloneqq \frac{1}{2} \left[(id,-f)_{\#} \mu + (id,f)_{\#} \mu\right] \in \Psp_2(\T\Sp)_{\mu}^0.
\end{align}
Note that $(\pi_x,-\pi_v)_{\#} \gamma_f = \gamma_f$. For any $f,g \in L^2_{\mu}(\Sp;\mathbb{R}^d)$, the transport plan $\frac{1}{2}\left[(id,f,g)_{\#} \mu + (id,-f,-g)_{\#} \mu\right]$ provides the estimate $W_{\mu}(\gamma_f,\gamma_g) \leqslant \|f - g\|_{L^2_{\mu}}$. 
\filler{Our interest for such fields stems from the following lemma.}

\begin{lem}\label{res:orthGamma}
	Let  $f \in L^2_{\mu}(\Sp;\mathbb{R}^d)$. A centred measure field $\xi$ is orthogonal to $\gamma_{f}$ if and only if for some (thus any) measurable $f^{\bdot} : \Sp \to \mathbb{R}^d$ in the equivalence class $f$, there holds $\left<v, f^{\bdot}(x)\right> = 0$ for $\xi-$almost every $(x,v)$.
\end{lem}

\begin{proof}
	If $\left<v,f^{\bdot}(x)\right> = 0$ for $\xi-$almost all $(x,v)$, then any transport plan between $\xi$ and $\gamma$ is concentrated on $(x,v,w)$ with $w = \pm f^{\bdot}(x)$, and there must hold $\left<v,w\right> = 0$ almost everywhere. This shows that $\left<\xi,\gamma\right>_{\mu} = 0$. Conversely, assume that $\xi$ is centred and $\left<\xi,\gamma\right>_{\mu} = 0$. Let $(\xi_x)_{x \in \Sp}$ be a disintegration of $\xi$, that can be chosen such that $\int_v v d\xi_x = 0$ for all $x \in \Sp$. By \Cref{res:centredislocal}, there holds for $\mu-$almost every $x$ that
	\begin{align}\label{orthGamma:disint}
		\int_{(v,w) \in (\mathbb{R}^d)^2} \left<v,w\right> d\alpha_x = 0
		\qquad \text{for all } \alpha_x \in \Gamma\left(\xi_x,\frac{1}{2}\left[\delta_{f^{\bdot}(x)} + \delta_{(-f^{\bdot}(x))}\right]\right).
	\end{align}
	If we show that \Cref{orthGamma:disint} implies $\left<v,f^{\bdot}(x)\right> = 0$ for $\xi_x-$almost every $v$, then the equality $\int |\left<v,f^{\bdot}(x)\right>| d\xi = \int_{x} \int_{v} |\left<v,f^{\bdot}(x)\right>| d\xi_x(v) d\mu(x)$ will ensure that $\left<v,f^{\bdot}(x)\right> = 0$ for $\xi-$almost every $(x,v)$. In the rest of the proof, we simplify the notation by letting $\zeta \coloneqq \xi_x \in \Psp_2(\mathbb{R}^d)^0$ and $\overline{w} \coloneqq f^{\bdot}(x) \in \mathbb{R}^d$. 
	
	Assume by contradiction that there exists $\varepsilon_+ > 0$ such that $m_+ \coloneqq \zeta\left\{ \left<v,\overline{w}\right> > \varepsilon_+ \right\} > 0$. Then, since $\int_{v} \left<v,\overline{w}\right> d\zeta = 0$, there must exist $\varepsilon_- > 0$ such that $m_- \coloneqq \zeta\left\{ \left<v,\overline{w}\right> < - \varepsilon_- \right\} > 0$. Let $m \coloneqq \min(m_-,m_+) > 0$. Construct a transport plan $\alpha \in \Gamma(\zeta,\frac{1}{2}\left[\delta_{-\overline{w}} + \delta_{\overline{w}}\right])$ by sending a mass $m/m_+$ from $\zeta_+ \coloneqq \zeta \resmes \{\left<v,\overline{w}\right> > \varepsilon_{+}\}$ to $\overline{w}$, a mass $m/m_-$ from $\zeta_{-} \coloneqq \zeta \resmes \{\left<v,\overline{w}\right> < -\varepsilon_-\}$ to $- \overline{w}$, and split the rest evenly; explicitly,
	\begin{align*}
		\alpha 
		\coloneqq m \left[\left(\frac{1}{m_+} \zeta_+\right) \otimes \delta_{\overline{w}} + \left(\frac{1}{m_-} \zeta_-\right) \otimes \delta_{-\overline{w}}\right] 
		+ (\zeta - \frac{m}{m^+} \zeta_+ - \frac{m}{m_-} \zeta_-) \otimes \left[\left(\frac{1}{2}-m\right)(\delta_{\overline{w}} + \delta_{-\overline{w}})\right].
	\end{align*}
	Then $\pi_{v \#} \alpha = \zeta$, and $\pi_{w \#} \alpha = \frac{1}{2} \left[\delta_{\overline{w}} + \delta_{-\overline{w}}\right]$. Moreover, 
	\begin{align*}
		\int_{(v,w) \in (\mathbb{R}^d)^2} \left<v,w\right> d\alpha
		= \frac{m}{m_+} \int_{v} \left<v,\overline{w}\right> d\zeta_+ + \frac{m}{m_-} \int_{v} \left<v,-\overline{w}\right> d\zeta_- + 0
		\geqslant m \varepsilon_+ + m \varepsilon_-
		> 0,
	\end{align*}
	against \Cref{orthGamma:disint}. Hence $\zeta\left\{ \left<v,\overline{w}\right> > \varepsilon_+ \right\} = 0$ for all $\varepsilon_+ > 0$, and since $\zeta$ is centred, we conclude.
\end{proof}

\filler{\Cref{res:orthGamma} will allow us to reduce a ``global'' orthogonality to a ``local'' one. We now come back to our $W_{\mu}-$closed, horizontally convex cone $\Set_{\mu}^0 \subset \Psp_2(\T\Sp)_{\mu}^0$, and introduce 
\begin{align*}
	\mathcal{F} \coloneqq \left\{ f \in L^2_{\mu}(\Sp;\mathbb{R}^d) \st \gamma_f \in \Set_{\mu}^0 \right\}.
\end{align*}
The following result shows that the set $\mathcal{F}$ is sufficient to characterize $(\Set_{\mu}^0)^{\perp 0}$. 
}

\begin{lem}\label{res:Fsuff}
	If $\xi \in \Psp_2(\T\Sp)_{\mu}^0$, then $\xi$ belongs to $(\Set_{\mu}^0)^{\perp 0}$ if and only if it is orthogonal to all $\gamma_f$ for $f \in \mathcal{F}$.
\end{lem}

\filler{The proof uses a construction that makes it quite verbose, but not complicated.}

\begin{proof}
	One implication being direct, we show that if $\left<\xi, \gamma_f\right>_{\mu} = 0$ for all $f \in \mathcal{F}$, then $\xi \in (\Set_{\mu}^0)^{\perp 0}$. This is equivalent to $\pi^{\mu} \xi = 0_{\mu}$, where $\pi^{\mu} \xi$ is the metric projection of $\xi$ on $\Set_{\mu}^0$, given by \Cref{res:projConvCone}. In particular, as $\gamma_f = (\pi_x,-\pi_v)_{\#} \gamma_f$ for any $f \in \mathcal{F}$, there holds $\left<\xi,\gamma_{f}\right>_{\mu} = \left<\pi^{\mu} \xi, \gamma_f\right>_{\mu}$. Assume by contradiction that $\pi^{\mu} \xi \neq 0_{\mu}$: we construct $f \in \mathcal{F}$ such that $\left<\pi^{\mu} \xi, \gamma_f\right>_{\mu} > 0$, against the assumption on $\xi$.
	
	First consider the constant vector fields $g_i(x) \equiv e_i$, where $e_i$ is the $i^{\text{th}}$ element of the canonical basis of $\mathbb{R}^d$. By \Cref{res:orthGamma}, if $\left<\pi^{\mu} \xi, \gamma_{g_i}\right>_{\mu} = 0$ for all $i$, then $\left<v,e_i\right> = 0$ for $\pi^{\mu} \xi-$almost every $(x,v)$ and each $i \in \llbracket 1,d\rrbracket$, so that $\pi^{\mu} \xi = 0_{\mu}$. By contradiction, there must be $i \in \llbracket 1,d\rrbracket$ such that $\left<\pi^{\mu} \xi, \gamma_{g_i}\right>_{\mu} > 0$. 
	
	Pick $\alpha \in \Gamma_{\mu,o}(\pi^{\mu} \xi, \gamma_{g_i})$. Then $\alpha = \frac{1}{2} (\pi_x,\pi_v,-e_i)_{\#} \zeta^{0}_- + \frac{1}{2} (\pi_x,\pi_v,e_i)_{\#} \zeta^{0}_+$ for some $\zeta^{0}_{\pm} \in \Psp_2(\T\Sp)_{\mu}$ such that $\pi^{\mu} \xi = \frac{1}{2} \zeta^{0}_{+} + \frac{1}{2} \zeta^{0}_- \in \Set_{\mu}^0$. Construct inductively a sequence $(\zeta^{k}_{+},\zeta^{k}_{-})_{k \in \mathbb{N}}$ as follows; assuming that $\frac{1}{2} \zeta^{k}_{+} + \frac{1}{2} \zeta^{k}_-$ belongs to $\Set_{\mu}^0$, consider the transport plan $\beta \coloneqq \frac{1}{2} \zeta^{k}_+ \otimes_{\mu} \zeta^{k}_{+} + \frac{1}{2} \zeta^{k}_- \otimes_{\mu} \zeta^{k}_{-}$, where for any $\zeta \in \Psp_2(\T\Sp)_{\mu}$, the ``pointwise product plan'' is defined as
	\begin{align*}
		\zeta \otimes_{\mu} \zeta \coloneqq \int_{x \in \Sp} \left[\delta_x \otimes \zeta_x \otimes \zeta_x\right] d\mu(x) \in \Psp_2(\T^2\Sp)_{\mu}.
	\end{align*}
	Let $\zeta^{k+1}_{\pm} \coloneqq \left(\pi_x, \frac{1}{2} \left[\pi_v + \pi_w\right]\right)_{\#} \zeta^{k}_{\pm} \otimes_{\mu} \zeta^{k}_{\pm}$. Then 
	\begin{align*}
		\frac{1}{2} \zeta^{k+1}_{+} + \frac{1}{2} \zeta^{k+1}_-
		= \Big(\pi_x,\frac{\pi_v+\pi_w}{2} \Big)_{\#} \beta
		\in \Set_{\mu}^0
	\end{align*}
	by horizontal convexity. The barycenter of each $\zeta^{k}_{\pm}$ is preserved along the sequence: indeed, for any $\varphi \in \mathcal{C}(\T\Sp;\mathbb{R})$ linear in its second argument and with quadratic growth,
	\begin{align*}
		\int_{(x,v)} \varphi(x,v) d\zeta^{k+1}_{+}
		= \int_{(x,v,w)} \varphi\left(x,\frac{v+w}{2}\right) d\left[\zeta^{k}_{+} \otimes_{\mu} \zeta^{k}_{+}\right]
		= \frac{1}{2} \int_{(x,v)} \varphi(x,v) d\zeta^{k}_{+} +  \frac{1}{2} \int_{(x,w)} \varphi(x,w) d\zeta^{k}_{+},
	\end{align*} 
	hence $\bary{\zeta^{k+1}_{\pm}}{} = \bary{\zeta^{k}_{\pm}}{} \eqqcolon f_{\pm}$. In particular, $\frac{1}{2} f_{+} + \frac{1}{2} f_- = \bary{\frac{1}{2} \zeta^{0}_{+} + \frac{1}{2} \zeta^{0}_{-}}{} = \bary{\pi^{\mu} \xi}{} = 0$, so that $f_- = - f_{+}$. Now, the sequence $(\zeta_{+}^{k})_{k \in \mathbb{N}}$ converges to $(id,f_{+})_{\#} \mu$: indeed, 
	\begin{align*}
		W_{\mu}^2\left(\zeta^{k+1}_{+}, (id,f_{+})_{\#} \mu\right)
		&= \int_{(x,v,w)} \left|\frac{v+w}{2} - f_{+}(x)\right|^2 d\left[\zeta^{k}_{+} \otimes_{\mu} \zeta^{k}_{+}\right] \\
		&= 2 \times \frac{1}{4} \int_{(x,v)} |v-f_+(x)|^2 d\zeta^{k}_+ + 2 \int_{(x,v,w)} \left<\frac{v - f_+(x)}{2}, \frac{w - f_+(x)}{2}\right> d\left[\zeta^{k}_{+} \otimes_{\mu} \zeta^{k}_{+}\right] \\
		&= \frac{1}{2} W_{\mu}^{2}\left(\zeta^{k}_{+}, (id,f_{+})_{\#} \mu\right) + 0,
	\end{align*}
	where we used that $\Gamma_{\mu}(\xi,(id,g)_{\#} \mu)$ always reduces to a singleton whenever $g \in L^2_{\mu}$, and the definition of the pointwise product measure. The same argument implies that $\zeta^k_{-} \to_k (id,f_{-})_{\#} \mu = (id,- f_{+})_{\#} \mu$ with respect to $W_{\mu}$. Consequently, the sequence $\left(\frac{1}{2} \zeta^k_{+} + \frac{1}{2} \zeta^k_-\right)_{k \in \mathbb{N}} \subset \Set_{\mu}^0$ converges with respect to $W_{\mu}$ towards the centred field $\frac{1}{2} (id,f_{+})_{\#} \mu + \frac{1}{2} (id,f_{-})_{\#} \mu = \gamma_{f_+}$, which must belong to $\Set_{\mu}^0$. Now, recalling the definition of $\zeta^0_{\pm}$, there holds
	\begin{align*}
		0 
		< \left<\pi^{\mu} \xi, \gamma_{g_i}\right>_{\mu}
		= \frac{1}{2} \int_{(x,v)} \left<v, - e_i\right> d\zeta^0_- + \frac{1}{2} \int_{(x,v)} \left<v, e_i\right> d\zeta^0_+
		= \int \left<f_+(x), e_i\right> d\mu,
	\end{align*}
	so that $\|f_{+}\|_{\mu} > 0$; on the other hand, using that $\left<\zeta,(id,g)_{\#}\right>_{\mu} = \left<\bary{\zeta}{},g\right>_{L^2_{\mu}}$ for any measure field $\zeta \in \Psp_2(\T\Sp)_{\mu}$ and vector field $g \in L^2_{\mu}(\Sp;\mathbb{R}^d)$,
	\begin{align*}
		\left<\pi^{\mu} \xi, \gamma_{f_+}\right>_{\mu}
		\geqslant \frac{1}{2} \left<\zeta^{0}_+, (id,f_{+})_{\#} \mu\right>_{\mu} + \frac{1}{2} \left<\zeta^{0}_-, (id,f_{-})_{\#} \mu\right>_{\mu}
		= \frac{1}{2} \|f_{+}\|_{L^2_{\mu}}^2 + \frac{1}{2} \|f_{-}\|_{L^2_{\mu}}^2
		= \|f_+\|_{L^2_{\mu}}^2
		> 0.
	\end{align*}
	In conclusion, if $\pi^{\mu} \xi \neq 0_{\mu}$, we constructed $f_+ \in \mathcal{F}$ such that $0 < \left<\pi^{\mu} \xi, \gamma_{f_+}\right>_{\mu} = \left<\xi, \gamma_f\right>$, in contradiction with the assumption. Hence $\pi^{\mu} \xi = 0$, and $\xi \in (\Sol_{\mu}^0)^{\perp 0}$. 	
\end{proof}

\subsection{Characterization by a Grassmannian section}

\filler{We can now turn to the proof of \Cref{res:charSet0}. Our aim is to construct a Grassmannian section $D : \Sp \rightrightarrows \mathbb{R}^d$ such that $\xi \in \Set_{\mu}^0$ if and only if $\xi$ is centred and concentrated on $\graph D$.}

\begin{proof}[Proof of \Cref{res:charSet0}]
	Assume first that a Grassmannian section $D : \Sp \rightrightarrows \mathbb{R}^d$ is given, and let $\Set_{\mu}^0$ be the set of centred measure fields concentrated on $\graph D$. Clearly, $\Set_{\mu}^0$ is a convex cone of centred measure fields. Moreover, for each $x \in \Sp$, the set of measures in $\Psp_2(\T_x\Sp) \simeq \Psp_2(\mathbb{R}^d)$ which are concentrated on the closed set $D(x)$ is closed with respect to the Wasserstein distance. Hence $\Set_{\mu}^0$ is closed with respect to $W_{\mu} \simeq L^2_{\mu}(\Sp;\left(\Psp_2(\mathbb{R}^d);W\right))$.

	Let now $\Set_{\mu}^0$ be a $W_{\mu}-$closed convex cone of centred measure fields. 
	By \Cref{res:biorth}, $\Set_{\mu}^0 = ((\Set_{\mu}^0)^{\perp 0})^{\perp 0}$.
	Let $\mathcal{F}_{\perp} \coloneqq \left\{ f \in L^2_{\mu}(\Sp;\mathbb{R}^d) \st \gamma_f \in (\Set_{\mu}^0)^{\perp 0} \right\}$. Since $W_{\mu}(\gamma_f,\gamma_g) \leqslant \|f-g\|_{L^2_{\mu}}$ and $(\Set_{\mu}^0)^{\perp 0}$ is $W_{\mu}-$closed, the set $\mathcal{F}_{\perp}$ is closed in $L^2_{\mu}$, hence separable. Consider a countable dense set $(f_n)_{n \in \mathbb{N}} \subset \mathcal{F}_{\perp}$. From \Cref{res:Fsuff} and the continuity of the scalar product, $\xi \in \Set_{\mu}^0 = ((\Set_{\mu}^0)^{\perp 0})^{\perp 0}$ if and only if $\left<\xi,\gamma_{f_n}\right>_{\mu} = 0$ for all $n \in \mathbb{N}$. By \Cref{res:orthGamma}, the latter condition is equivalent to $\left<v,f_n^{\bdot}(x)\right> = 0$ for $\xi-$almost all $(x,v)$, where $f_n^{\bdot} : \Sp \to \mathbb{R}^d$ is a measurable map in the $L^2_{\mu}-$equivalence class $f_n$. For every $x \in \Sp$, define 
	\begin{align*}
		D(x) \coloneqq \left\{ v \in \mathbb{R}^d \st \left<v,f^{\bdot}_n(x)\right> = 0 \text{ for all } n \in \mathbb{N} \right\}.
	\end{align*}
	The application $D$ depends on the precise choice of $(f_n^{\bdot})_n$, but only up to a $\mu-$negligible subset. Each $D(x)$ is a vector space, and by \cite[Theorem 3.(e)]{rockafellarMeasurableDependenceConvex1969}, $D$ is measurable as a multivalued application. If $\xi$ is concentrated on the graph of $D$, then $\left<v,f_n^{\bdot}(x)\right> = 0$ for all $n$ $\xi-$almost everywhere. Conversely, if for any $n$, there exists $B^n$ such that $\xi(B^n) = 0$ and $\left<v,f_n^{\bdot}(x)\right> = 0$ for any $(x,v) \notin B^n$, then $B \coloneqq \bigcup_n B^n$ stays $\xi-$negligible, and $v \in D(x)$ for any $x \notin B$. Hence $\xi \in \Set_{\mu}^0$ if and only if $\xi(\graph D) = 1$, as claimed.
\end{proof}

\filler{An interesting corollary of \Cref{res:charSet0} is that $\Set_{\mu}^0$ is closed with respect to the (weaker) topology of the Wasserstein distance over the tangent bundle $W_{\T\Sp}$, i.e. with cost $c((x,v),(y,w)) \coloneqq \sqrt{|x-y|^2 + |v-w|^2}$. Here, it is necessary to restrict to centred measure fields: the geometric tangent cone $\Tan_{\mu}$ is a $W_{\mu}-$closed, horizontally convex and (two-sided) cone, but is not closed with respect to $W_{\T\Sp}$ in general (an example can be found below Proposition 2.10 in \cite{aussedatStructureGeometricTangent2025}).}

\begin{cor}[$W_{\T\Sp}-$closedness of $\Set_{\mu}^0$]\label{res:Set0closed}
	Assume that $\Set_{\mu}^0$ satisfies \Cref{hyp:Set0}. Let $(\xi_n)_{n \in \mathbb{N}} \subset \Set_{\mu}^0$ be a sequence converging to $\xi \in \Psp_2(\T\Sp)_{\mu}$ with respect to $W_{\T\Sp}$. Then $\xi \in \Set_{\mu}^0$. 
\end{cor}

\begin{proof}
	By \cite[Theorem 6.9]{villaniOptimalTransport2009}, convergence with respect to $W_{\T\Sp}$ is equivalent to the convergence of $\int_{(x,v)} \varphi(x,v) d\xi_n \to_n \int_{(x,v)} \varphi(x,v) d\xi$ for any continuous and quadratically growing $\varphi : \T\Sp \to \mathbb{R}$. In consequence, $\int \varphi(x,\bary{\xi}{}(x)) d\mu = 0$ for any $\varphi$ linear with respect to its second argument, and $\bary{\xi}{} = 0$. By \Cref{res:DWTSpClosed}, the set of $\xi$ such that $\xi(\graph D) = 1$ is $W_{\T\Sp}-$closed, hence $\xi \in \Set_{\mu}^0$.
\end{proof}

\filler{The characterization of $\Set_{\mu}^0$ provides many examples of closed convex cones of centred fields; one just has to choose the map $D$. However, the following section focuses on two particular subsets of measure fields which are not \emph{a priori} constructed from such maps, but can be proved to be closed and horizontally convex, yielding an additional structure.}

\section{Tangent and solenoidal measure fields}\label{sec:tansol}

\filler{We introduce the geometric tangent cone $\Tan_{\mu}$ in its classical definition, as well as its metric orthogonal, and immediately restrict our attention to their centred subsets.}

\begin{definition}[$\Tan_{\mu}^0$ and $\Sol_{\mu}^0$]\label{def:TanSol0}
	The geometric tangent cone is defined as 
	\begin{align*}
		\Tan_{\mu} 
		\coloneqq \overline{\left\{ \lambda \cdot \xi \st \lambda \geqslant 0, \ (\pi_x,\pi_x+\pi_v)_{\#} \xi \text{ is optimal between its marginals} \right\}}^{W_{\mu}}
		\subset \Psp_2(\T\Sp)_{\mu}.
	\end{align*}
	The solenoidal cone $\Sol_{\mu}$ is defined as $(\Tan_{\mu})^{\perp}$, i.e. 
	\begin{align*}
		\Sol_{\mu} 
		\coloneqq \left\{ \zeta \in \Psp_2(\T\Sp)_{\mu} \st \left<\xi,\zeta\right>_{\mu} = 0 \text{ for all } \xi \in \Tan_{\mu} \right\}.
	\end{align*}
	The centred cones $\Tan_{\mu}^0$ and $\Sol_{\mu}^0$ are defined as the intersections of $\Tan_{\mu},\Sol_{\mu}$ with $\Psp_2(\T\Sp)_{\mu}^0$.
\end{definition}

\begin{rem}[Alternative definition]\label{rem:barycentreddecomp}
	One easily shows that for any $\xi,\zeta \in \Psp_2(\T\Sp)_{\mu}$, there holds
	\begin{align}\label{barycentreddecomp:splitps}
		\left<\xi,\zeta\right>_{\mu} = \left<\xi^0,\zeta^0\right>_{\mu} + \left<\bary{\xi}{}, \bary{\zeta}{}\right>_{L^2_{\mu}},
	\end{align}
	where $\xi^0,\zeta^0$ are the centred components of $\xi,\zeta$, i.e. $\xi^0 = (\pi_x, \pi_v - \bary{\xi}{})_{\#} \xi$. Therefore the sets $\Tan_{\mu}^0$ and $\Sol_{\mu}^0$ coincide with the sets of metric projections of $\Tan_{\mu}, \Sol_{\mu}$ over the closed, convex cone $\Psp_2(\T\Sp)_{\mu}^0$. The orthogonal decomposition \Cref{barycentreddecomp:splitps} also implies that $\xi$ is tangent (resp. $\zeta$ solenoidal) if and only if $\xi^0$ and $(id,\bary{\xi}{})_{\#} \mu$ are tangent (resp. $\zeta^0$ and $(id,\bary{\zeta}{})_{\#} \mu$ solenoidal).
\end{rem}

\filler{Our first result on $\Tan_{\mu}^0$ and $\Sol_{\mu}^0$ is an application of \Cref{res:charSet0}. If $D \rightrightarrows \mathbb{R}^d$ is a Grassmannian section, denote by $D^{\perp}$ the application such that $D^{\perp}(x)$ is the orthogonal complement of $D(x)$ in $\mathbb{R}^d$.}

\begin{cor}\label{res:charTanSol0}
	Let $\mu \in \Psp_2(\Sp)$. There exists a Grassmannian section $D : \Sp \rightrightarrows \mathbb{R}^d$ such that 
	\begin{align}\label{charTanSol0:statement}
		\begin{matrix}
			\zeta \in \Sol_{\mu}^{0} &\iff & \big[\zeta \in \Psp_2(\T\Sp)_{\mu}^0 & \text{and} & \zeta(\graph D) = 1\big], \vspace{5pt}\\
			\xi \in \Tan_{\mu}^{0} &\iff & \big[\xi \in \Psp_2(\T\Sp)_{\mu}^0 & \text{and} & \xi(\graph D^{\perp}) = 1\big].
		\end{matrix}
	\end{align}
\end{cor}

In the sequel, the application $D$ will often be denoted $D^{\Sol}$, and $D^{\perp}$ will be denoted $D^{\Tan}$.

\begin{proof}
	The fact that $\Tan_{\mu}^0$ is a $W_{\mu}-$closed positive cone of centred fields is direct from the definition. The fact that it is horizontally convex is proved in \cite[Proposition 4.25]{gigliGeometrySpaceProbability2008}, first by considering optimal plans and using cyclical monotonicity, then by approximation. The decomposition of the metric scalar product in \Cref{barycentreddecomp:splitps} yields that $\Sol_{\mu}^{0} = (\Tan_{\mu}^0)^{\perp 0}$, so that by \Cref{res:orthGamma}, $\Sol_{\mu}^0$ is also a closed convex cone of centred measure fields. Hence the existence of $D$ characterizing $\Sol_{\mu}^0$ as in \Cref{charTanSol0:statement} follows by \Cref{res:charSet0}. The same result yields a Grassmannian section $\widetilde{D} : \Sp \rightrightarrows \mathbb{R}^d$ characterizing $\Tan_{\mu}^0$; to conclude, there only stays to show that $\widetilde{D} = D^{\perp}$ $\mu-$almost everywhere. But this is implied by the following equivalences for $\xi$ centred:
	\begin{align*}
		&&&\xi \text{ concentrated on } (x,v) \text{ such that } v \in \widetilde{D}(x) \\
		\text{\small\Cref{res:charSet0}} &&\iff \quad& \xi \in \Tan_{\mu}^0 \\
		\text{\small\Cref{res:biorth,res:Fsuff}} && \iff \quad& \left<\xi,\gamma_{f}\right>_{\mu} = 0 \text{ for any } f \in L^2_{\mu} \text{ such that } \gamma_f \in \Sol_{\mu}^0 \\
		\text{\small\Cref{res:charSet0}} &&\iff \quad& \xi \text{ concentrated on } (x,v) \text{ such that } \left<v, f^{\bdot}(x)\right> = 0 \text{ for all } f^{\bdot} \underset{L^2_{\mu}}{\subset} D,
	\end{align*}
	where in the last line, $f^{\bdot}$ ranges in all measurable functions such that $f^{\bdot}(x) \in D(x)$ for all $x \in \Sp$, and $\int |f^{\bdot}|^2 d\mu < \infty$. The fact that sufficiently many of such $f^{\bdot}$ can be found is given by the Castaing representation of the measurable map $D$, for instance in \cite[Theorem 3.(d)]{rockafellarMeasurableDependenceConvex1969}.
\end{proof}

\filler{By \Cref{res:Set0closed}, both $\Tan_{\mu}^0$ and $\Sol_{\mu}^0$ are closed with respect to the Wasserstein distance $W_{\T\Sp}$ on the tangent bundle. In addition, the set of solenoidal measure fields (possibly with nonzero barycenter) is closed in the same topology.}

\begin{cor}[$\Sol_{\mu}$ is $W_{\T\Sp}-$closed]
	Let $\mu \in \Psp_2(\Sp)$. The set $\Sol_{\mu}$ is closed with respect to $W_{\T\Sp}$. 
\end{cor}

\begin{proof}
	We first show that any $\zeta \in \Sol_{\mu}$ is concentrated on $\graph D$. By \Cref{rem:barycentreddecomp}, if $\zeta \in \Sol_{\mu}$, then $g \coloneqq \bary{\zeta}{} \in L^2_{\mu}(\Sp;\mathbb{R}^d)$ is such that $(id,g)_{\#} \mu \in \Sol_{\mu}$. By \cite[Lemmata 2.4 and 2.7]{aussedatStructureGeometricTangent2025}, both measure fields $(id,-g)_{\#} \mu$ and $\frac{1}{2} \left[(id,g)_{\#} \mu + (id,-g)_{\#} \mu\right]$ are solenoidal. Applying \Cref{res:charTanSol0} to the latter, we get that $g(x) \in D^{\Sol}(x)$ for $\mu-$almost every $x$. As the centred field $\zeta^0 \coloneqq (\pi_x,\pi_v-g(\pi_x))_{\#} \zeta$ is solenoidal, it is also concentrated on $\graph D$, thus so is $\zeta = (\pi_x,\pi_v + g(\pi_x))_{\#} \zeta^0$.
	
	Let now $(\zeta_n)_{n \in \mathbb{N}} \subset \Sol_{\mu}$ be a Cauchy sequence with respect to $W_{\T\Sp}$, and $\zeta \in \Psp_2(\T\Sp)_{\mu}$ be its limit. By \Cref{res:DWTSpClosed}, $\zeta(\graph D) = 1$. Consequently, the centred component $\zeta^0 \coloneqq (\pi_x,\pi_v-\bary{\zeta}{}(\pi_x))_{\#} \zeta$ is concentrated on $\graph D$, hence solenoidal. To show that $\bary{\zeta}{}$ also induces a solenoidal field, it is enough to prove that $\left<f,\bary{\zeta}{}\right>_{L^2_{\mu}} = 0$ for any $f \in L^2_{\mu}(\Sp;\mathbb{R}^d)$ such that $(id,f)_{\#} \mu$ is tangent (by \Cref{barycentreddecomp:splitps}). In addition, we may let $f = \nabla \varphi$ for $\varphi \in \mathcal{C}^{\infty}_c(\Sp;\mathbb{R})$ \cite[Theorem 4.14]{gigliGeometrySpaceProbability2008}. In particular, $(x,v) \mapsto \left<f(x),v\right>$ is continuous and has quadratic growth, so that
	\begin{align*}
		\left<f,\bary{\zeta}{}\right>_{L^2_{\mu}}
		= \int_{(x,v) \in \T\Sp} \left<f(x),v\right> d\zeta 
		= \lim_{n \to \infty} \int_{(x,v) \in \T\Sp} \left<f(x),v\right> d\zeta_n
		= \lim_{n \to \infty} \left<f, \bary{\zeta_n}{}\right>_{L^2_{\mu}}
		= 0.
	\end{align*}
	In conclusion, both $\zeta^0$ and $(id,\bary{\zeta}{})_{\#} \mu$ belong to $\Sol_{\mu}$, so that $\zeta$ is solenoidal.
\end{proof}

\filler{One may be tricked into thinking that $\Tan_{\mu}$ should be $W_{\T\Sp}-$closed by the same arguments. This is not correct: it does not hold that $\xi(\graph D^{\Tan}) = 1$ for any tangent $\xi$. It \emph{does} hold that 
\begin{itemize}
\item $\Tan_{\mu}^0$ is $W_{\T\Sp}-$closed (by \Cref{res:Set0closed}), 
\item the set of $f \in L^2_{\mu}(\Sp;\mathbb{R}^d)$ such that $(id,f)_{\#} \mu \in \Tan_{\mu}$ is weakly closed in $L^2_{\mu}$ (by the previous proof), 
\end{itemize}
but the latter set of vector fields is not strongly closed in $L^2_{\mu}$, and the oscillations captured by the $W_{\T\Sp}-$limit may get out of $D^{\Tan}$. On the other hand, any map inducing a solenoidal field is already valued in $D^{\Sol}$, so that oscillations can only produce measure fields that are again concentrated on $D^{\Sol}$.
}

\subsection{Stability with respect to restriction}

\filler{The aim of this section is to investigate how $\Tan_{\mu}^0$ and $\Sol_{\mu}^0$ depend on the local properties of the underlying measure $\mu$. We start by showing that the centred solenoidal spaces are stable by restriction of measures. This is valid only on centred measure fields; \Cref{rem:centredNecessary} below provides a counterexample in the case of map-induced fields.}

Given a measurable set $A \subset \Sp$, denote by $\mu \resmes A$ the measure given by $(\mu \resmes A)(B) \coloneqq \mu(A \cap B)$ for any measurable $B \subset \Sp$, and by $\T A \subset \T\Sp$ the set $\left\{ (x,v) \st x \in A, \ v \in \T_x\Sp \right\}$. 

\begin{prop}[Restriction of centred solenoidal measure fields]\label{res:restr:Sol}
	Let $\mu \in \Psp_2(\Sp)$ and $A \subset \Sp$ be a measurable set such that $\mu(A) \in (0,1)$. Denote $\mu_{A} \coloneqq (\mu \resmes A)/\mu(A)$. Then 
	\begin{align*}
		\Sol_{\mu_A}^0 = \left\{ (\zeta \resmes \T A) / \mu(A) \st \zeta \in \Sol_{\mu}^0 \right\}.
	\end{align*}
\end{prop}

\filler{\Cref{res:restr:Sol} relies on the following intermediate results, whose proofs are delayed to the \Cref{sec:appendix}. The statements are formulated for $\Sol_{\mu}$ instead of $\Sol_{\mu}^0$, since the centred character does not intervene there.} 

\begin{lem}\label{res:extensionOpt}
	Let $\mu = (1-\lambda) \mu_1 + \lambda \mu_2 \in \Psp_2(\Sp)$ for $\mu_i \in \Psp_2(\Sp)$ and $\lambda \in [0,1]$. Let $\eta \in \Psp_2(\T\Sp)_{\mu_1}$ be the velocity of a geodesic, with $\nu \coloneqq (\pi_x+\pi_v)_{\#} \eta$ compactly supported. There exists $\gamma \in \Psp_2(\T\Sp)_{\mu_2}$ such that 
	\begin{align}\label{extensionOpt:statement}
		\xi = (1-\lambda) \eta + \lambda \gamma
	\end{align} 
	is the velocity of a geodesic issued from $\mu$.
\end{lem}

\filler{In the above result, the measure $\nu$ is compactly supported. This is sharp; take for instance $\mu_1 = \delta_0$, $\mu_2 = \delta_1$ and $\nu$ a Gaussian measure in dimension one. There is only one $\eta$ such that $(\pi_x,\pi_x+\pi_v)_{\#} \eta$ is optimal between $\mu_1$ and $\nu$, and its support is not bounded on the velocity variable. To construct a $\gamma$ satisfying \Cref{extensionOpt:statement}, one should be able to ensure that for any $(x,v) \in \supp \eta$ and $(y,w) \in \supp \gamma$, the monotonicity condition $|v|^2 +|w|^2 \leqslant |x+v-y|^2 + |y+w-x|^2$ holds. Equivalently, using that $x = 0$ and $y = 1$, one should have $2 v - 1 \leqslant 1 + 2 w$; since $v$ is arbitrarily large, such $w$ cannot be finite. In consequence, to be able to use \Cref{res:extensionOpt}, we characterize $\Sol_{\mu}$ by orthogonality with respect to velocities going towards compactly supported measures.}

\begin{lem}\label{res:SolIfOrthCompSupp}
	Let $\zeta \in \Psp_2(\T\Sp)_{\mu}$ satisfy $\left<\zeta,(\pi_x,\pi_y-\pi_x)_{\#} \gamma\right>_{\mu} = 0$ for any $\gamma \in \Gamma_o(\mu,\nu)$ with $\nu \in \Psp_2(\Sp)$ compactly supported. Then $\zeta \in \Sol_{\mu}$. 
\end{lem}

\filler{We now turn to our original claim. In the proof, we use the formula given by \cite[Proposition 4.2]{gigliGeometrySpaceProbability2008} to deduce a Chasles relation for the metric scalar product, stating that for $\mu \in \Psp_2(\Sp)$ and $A \subset \Sp$ measurable, 
\begin{align*}
	\left<\xi,\zeta\right>_{\mu}
	= \int_{x \in \Sp} \left<\xi_x, \zeta_x\right>_{\delta_x} d\mu(x)
	= \int_{x \in A} \left<\xi_x, \zeta_x\right>_{\delta_x} d\mu(x) + \int_{x \in A^c} \left<\xi_x, \zeta_x\right>_{\delta_x} d\mu(x).
\end{align*}
}

\begin{proof}[Proof of \Cref{res:restr:Sol}]
	We proceed by double inclusion. Consider $\zeta \in \Sol_{\mu}^0$, and let $\zeta_A \coloneqq (\zeta \, \resmes \T A) / \mu(A)$ and $\zeta_{A^c} \coloneqq (\zeta \resmes \T A^c) / \mu(A^c)$. By \Cref{res:SolIfOrthCompSupp}, it suffices to show that $\left<\zeta_A, \xi^A\right>_{\mu_A} = 0$ for any $\xi^A \in \Psp_2(\T\Sp)_{\mu_A}$ which induces a geodesic between $\mu_A$ and a compactly supported measure. Consider such a $\xi^A$. Applying \Cref{res:extensionOpt} with $\mu_1 = \mu_A$, $\mu_2 = \mu_{A^c}$ and $1-\lambda = \mu(A)$, we obtain a measure field $\xi^{A^c} \in \Psp_2(\T\Sp)_{\mu_{A^c}}$ such that $\xi = \mu(A) \xi^A + \mu(A^c) \xi^{A^c}$ belongs to $\Tan_{\mu}$. As $\mu_{A}$ and $\mu_{A^c}$ are mutually singular, there holds $\xi^A = (\xi \resmes \T A) / \mu(A)$. Hence
	\begin{align}\label{restr:Sol:Chasles}
		0 
		= \left<\zeta, \xi\right>_{\mu}
		= \mu(A) \left<\zeta_A, \xi^A\right>_{\mu_A} + \mu(A^c) \big<\zeta_{A^c}, \xi^{A^c}\big>_{\mu_{A^c}}
		\geqslant \mu(A) \left<\zeta_A, \xi^A\right>_{\mu_A}
		\geqslant 0,
	\end{align}
	where we used that $\zeta_A$ and $\zeta_{A^c}$ are centred (see \Cref{res:centredislocal}). Hence $\left<\zeta_A,\xi^A\right>_{\mu_A} = 0$, and $\zeta_A \in \Sol_{\mu_A}^0$. 
	
	Conversely, let $\zeta^A \in \Sol_{\mu_A}^0$, and consider $\zeta \coloneqq \mu(A) \zeta^A + \mu(A^c) 0_{\mu_{A^c}}$. In particular, $\zeta^A = (\zeta \resmes \T A)/\mu(A)$. Let $\xi \in \Psp_2(\T\Sp)_{\mu}$ be the velocity of a geodesic. By restriction of optimality \cite[Theorem 4.6]{villaniOptimalTransport2009}, $\xi_A \coloneqq (\xi \resmes \T A) / \mu(A)$ is also the velocity of a geodesic; moreover, by Chasles,
	\begin{align*}
		\left<\zeta,\xi\right>_{\mu}
		= \mu(A) \left<\zeta^A, \xi_A\right>_{\mu_A} + \mu(A^c) \left<0_{\mu_{A^c}}, \xi_{A^c}\right>_{\mu_{A^c}}
		= \mu(A) \left<\zeta^A, \xi_A\right>_{\mu_A}
		= 0.
	\end{align*}
	Hence $\zeta \in \Sol_{\mu}^0$, and any solenoidal measure field in $\Sol_{\mu_A}^0$ writes as the restriction of an element of $\Sol_{\mu}^0$.	
\end{proof}

\filler{As a corollary, we deduce the corresponding statement on the centred tangent cone.}

\begin{cor}[Restriction of centred tangent measure fields]\label{res:restr:Tan}
	With the same notations as in \Cref{res:restr:Sol}, 
	\begin{align*}
		\Tan_{\mu_A}^0 = \left\{ (\xi \resmes \T A) / \mu(A) \st \xi \in \Tan_{\mu}^0 \right\}.
	\end{align*}
\end{cor}

\begin{proof}
	Let first $\xi \in \Tan_{\mu}^0$, and denote $\xi_A \coloneqq (\xi \resmes \T A) / \mu(A)$. By \Cref{res:restr:Sol}, any $\zeta^A \in \Sol_{\mu_A}^0$ writes as $(\zeta \resmes \T A) / \mu(A)$ for some $\zeta \in \Sol_{\mu}^0$. Hence, using Chasles as in \Cref{restr:Sol:Chasles}, $0 = \left<\xi,\zeta\right>_{\mu} \geqslant \mu(A) \left<\xi_A, \zeta^A\right>_{\mu_A} \geqslant 0$, so that $\xi_A \in \Tan_{\mu_A}^0$. Conversely, if $\xi^A \in \Tan_{\mu_A}^0$, define $\xi \coloneqq \mu(A) \xi^A + \mu(A^c) 0_{\mu_{A^c}} \in \Psp_2(\T\Sp)_{\mu}^0$. For any $\zeta \in \Sol_{\mu}^0$, one has $\left<\xi,\zeta\right>_{\mu} = \mu(A) \left<\xi^A,\zeta_A\right>_{\mu_A}$, where $\zeta_A \coloneqq (\zeta \resmes \T A) / \mu(A)$ belongs to $\Sol_{\mu_A}^0$ by \Cref{res:restr:Sol}. Hence $\left<\xi,\zeta\right>_{\mu} = 0$, and $\xi \in \Tan_{\mu}^0$, completing the proof.
\end{proof}

\filler{Here we highlight that the argument is not exactly symmetric between $\Tan_{\mu}^0$ and $\Sol_{\mu}^0$; the difficulty lies in \Cref{res:extensionOpt}, where an \emph{optimal} plan attached to a measure is ``extended'' to an optimal plan attached to another measure. Despite many attempts, the author could not find a direct proof of an extension result for \emph{tangent} measure fields: when letting the optimal time decrease to 0, there is no guarantee that the narrow/Wasserstein limit stays tangent. However, optimal plans are sufficient to characterize solenoidal measure fields, so we can first prove the restriction on $\Sol_{\mu}^0$, then mirror it on $\Tan_{\mu}^0$.}

\begin{rem}[Necessity of the centred assumption]\label{rem:centredNecessary}
	Consider $\mu \in \Psp_2(\Sp)$ the 1-Hausdorff measure restricted to the unit square $S \coloneqq \partial [0,1]^2$. Parametrize $S$ by a constant-speed closed curve $\gamma : [0,1] \to \mathbb{R}^2$ rotating clockwise, and let $\zeta \coloneqq (\gamma,\dot{\gamma})_{\#} \mathcal{L}_{[0,1]}$. Then $\zeta$ is solenoidal; since it is induced by a map, this is equivalent to $\left<\xi,\zeta\right>_{\mu} = 0$ for any map-induced tangent $\xi$. Any such $\xi$ can be approximated arbitrarily well with respect to $W_{\mu}$ by $(id,\nabla \varphi)_{\#} \mu$ for some $\varphi \in \mathcal{C}^{\infty}_c(\Sp;\mathbb{R})$ \cite{gigliInverseImplicationBrenierMcCann2011}. As $\left<(id,\nabla \varphi)_{\#} \mu, \zeta\right>_{\mu} = \int_{[0,1]} \frac{d}{dt} \varphi \circ \gamma \, dt = 0$, the measure field $\zeta$ is solenoidal. 
	
	However, if $\nu \coloneqq 4 \, \mu \resmes [0,1] \times \{1\}$ is the (normalized) restriction of $\mu$ to the top side of the square, then the corresponding restriction $\varsigma \coloneqq (id,(1,0))_{\#} \nu$ belongs to $\Tan_{\nu}$, since it induces a geodesic. The reader may check that the centred fields $\frac{1}{2} \left[\zeta + (\pi_x,-\pi_v)_{\#} \zeta\right]$ and $\frac{1}{2} \left[\varsigma + (\pi_x,-\pi_v)_{\#} \varsigma\right]$ are both solenoidal.
\end{rem}

\subsection{Decomposition according to the dimension of splitting}\label{sec:decomp}

\filler{With the above material, we can now state and prove the first main result of the paper.}

\begin{theorem}\label{res:decomp}
	Let $\mu \in \Psp_2(\Sp)$. There exists a decomposition $\mu = \sum_{k=0}^d m_k \mu^k$, where $m_k \in [0,1]$ sum to one and $\mu^k \in \Psp_2(\Sp)$ are mutually singular measures, with the following properties. For each $k \in \llbracket 0,d \rrbracket$, there exists a Grassmannian section $D_k : \Sp \rightrightarrows \mathbb{R}^d$ such that $\dim D_k \equiv k$, and 
	\begin{enumerate}[label={\roman*)}]
	\item\label{item:decomp:sol:sum} $\zeta \in \Sol_{\mu}^0$ if and only if $\zeta = \sum_{k=0}^d m_k \zeta^k$ with $\zeta^k \in \Sol_{\mu^k}^0$ for $k \in \llbracket 0,d\rrbracket$. 
	\item\label{item:decomp:tan:sum} $\xi \in \Tan_{\mu}^0$ if and only if $\xi = \sum_{k=0}^d m_k \xi^k$ with $\xi^k \in \Tan_{\mu^k}^0$ for $k \in \llbracket 0,d\rrbracket$.
	\item\label{item:decomp:sol:con} If $m_k > 0$, $\zeta^k \in \Sol_{\mu^k}^0$ if and only if $\zeta^k$ is centred and concentrated on $\graph D_k$.
	\item\label{item:decomp:tan:con} If $m_k > 0$, $\xi^k \in \Tan_{\mu^k}^0$ if and only if $\xi^k$ is centred and concentrated on $\graph D_k^{\perp}$.
	\end{enumerate}
\end{theorem}

\filler{In addition, the measures $m_k \mu^k$ in the decomposition are unique. This directly follows from the explicit formula given in \Cref{res:charmuk}. \Cref{fig:decomp} provides a visual intuition supporting \Cref{res:decomp}.}

\begin{figure}[H]
	\centering
	\includegraphics{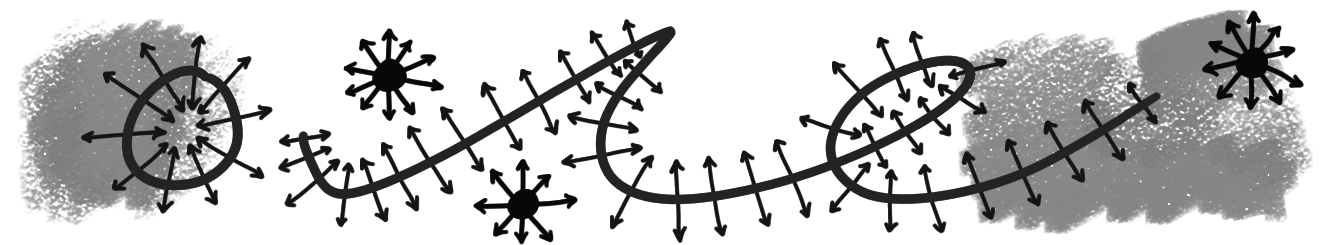}
	\caption{A measure $\mu \in \Psp_2(\mathbb{R}^2)$ and its Grassmannian section $D^{\Tan} = D^{\perp}$ characterizing $\Tan_{\mu}^0$.}\label{fig:decomp}
	\mysubcaption[0.94]{The measure $\mu^0$ is the sum of the three atoms, and $D^{\Tan}_0$ is $2-$dimensional. The measure $\mu^1$ is supported on a countable union of \DC{1} sets, which, in dimension 2, are graphs of DC functions up to permuting the axes. The direction of the one-dimensional Grassmannian section $D^{\Tan}_1$ is represented by the arrows, whose norm is irrelevant. The fact that $D^{\Tan}_1$ is orthogonal to the set over which $\mu^1$ is concentrated will be proved in \Cref{res:PeqD}. The remaining measure $\mu^2$ is transport-regular in the sense that it gives 0 mass to any \DC{1} set, and satisfies the conclusion of the Brenier-McCann theorem. In consequence, no optimal plan splits mass, and $D^{\Tan}_2$ is reduced to $\{0\}$.}
\end{figure}

\begin{proof}[Proof of \Cref{res:decomp}]
	By \Cref{res:charTanSol0}, there exist a Grassmannian section $D : \Sp \rightrightarrows \mathbb{R}^d$ such that $\zeta \in \Sol_{\mu}^0$ if and only if $\zeta$ is centred and concentrated on $\graph D$, and $\xi \in \Tan_{\mu}^0$ if and only if $\xi \in \Psp_2(\T\Sp)_{\mu}^0$ and $\xi$ is concentrated on $\graph D^{\perp}$. For each $k \in \llbracket 0,d\rrbracket$, define 
	\begin{align*}
		A_k \coloneqq \left\{ x \in \Sp \st \dim D(x) = k \right\}.
	\end{align*}
	Each set $A_k$ is measurable; indeed, by \cite[Theorem 3.(d)]{rockafellarMeasurableDependenceConvex1969}, there holds $D(x) = \overline{\conv \left\{ g_n(x) \st n \in \mathbb{N} \right\}}$ for some countable family $(g_n)_{n \in \mathbb{N}}$ of measurable applications. Then $A_k$ writes as the set of $x$ such that any choice of $k+1$ vectors $g_n(x)$ is linked, and there exists $k$ independent vectors $(g_{n_j}(x))_{j \in \llbracket 1,k\rrbracket}$, i.e. 
	\begin{align*}
		A_k = 
		\bigcap_{\sigma : \llbracket1,k+1\rrbracket \to \mathbb{N}} \left\{ {\det}_{k+1}\left(g_{\sigma_1}(x),\cdots,g_{\sigma(k+1)}(x)\right) = 0 \right\}
		\bigcap \bigcup_{\theta : \llbracket 1,k\rrbracket \to \mathbb{N}} \left\{ {\det}_{k}\left(g_{\theta(1)}(x),\cdots,g_{\theta(k)}(x)\right) \neq 0 \right\}.
	\end{align*}
	Here ${\det}_j : (\mathbb{R}^d)^j \to \mathbb{R}$ is the $j-$determinant, which is continuous. As each $g_n$ is measurable, so is $A_k$.
	
	If $m_k \coloneqq \mu(A_k) > 0$, define $\mu^k \coloneqq (\mu \resmes A^k) / \mu(A_k)$. Since $(A_k)_{k=0}^d$ is a partition of $\Sp$, the measures $\mu^k$ are mutually singular, and $\mu = \sum_{k=0}^d m_k \mu^k$. Let 
	\begin{align*}
		D_k(x) \coloneqq 
		\begin{cases}
			D(x) & x \in A_k, \\
			\vect \{e_1,\cdots,e_k\} & \text{otherwise}.
		\end{cases}
	\end{align*}
	Clearly, $D_k$ is measurable.
	We now show that the measures $\mu^k$ and the applications $D_k$ satisfy the claims. 
	
	Let $\zeta \in \Sol_{\mu}^0$, and write it as $\sum_{k=0}^d m_k \zeta^k$, where $m_k \zeta^k \coloneqq \zeta \, \resmes \T A_k$. In particular, $\zeta^k$ is centred. For each $k$ such that $m_k > 0$, \Cref{res:restr:Sol} yields that $\zeta^k \in \Sol_{\mu^k}^0$. This proves the first implication of point \Cref{item:decomp:sol:sum}.
	Conversely, let $(\zeta^k)_{k \in \llbracket 0,d \rrbracket}$ be a family of centred measure fields such that $\zeta^k \in \Sol_{\mu^k}^0$, and define $\zeta \coloneqq \sum_{k=0}^d m_k \zeta^k$. To show that $\zeta \in \Sol_{\mu}^0$, let $\xi \in \Tan_{\mu}^0$, decomposed as $\sum_{k=0}^d m_k \xi^k$. By Chasles, 
	\begin{align*}
		\left<\zeta,\xi\right>_{\mu}
		= \sum_{k = 0}^d m_k \left<\zeta^k, \xi^k\right>_{\mu^k}.
	\end{align*}
	By \Cref{res:restr:Tan}, each $\xi^k$ belongs to $\Tan_{\mu^k}^0$, so that every term of the sum is 0. Hence $\zeta \in (\Tan_{\mu}^0)^{\perp 0} = \Sol_{\mu}^0$. This proves \Cref{item:decomp:sol:sum}; the argument is completely symmetric for \Cref{item:decomp:tan:sum}, with \Cref{res:restr:Sol} in place of \Cref{res:restr:Tan}.
	
	We turn to point \Cref{item:decomp:sol:con}. If $\zeta^k \in \Sol_{\mu^k}^0$, then, by \Cref{item:decomp:sol:sum}, the measure field $\zeta \coloneqq m_k \zeta^k + \sum_{j=0, j \neq k}^d m_j 0_{\mu^j}$ belongs to $\Sol_{\mu}^0$. As such, it is concentrated on $\graph D$. Therefore so is $\zeta^k$, and since $D = D_k$ for $\mu^k-$almost every $x$, we get that $\zeta^k(\graph D_k) = 1$.
	On the other hand, let $\zeta^k \in \Psp_2(\T\Sp)_{\mu^k}^0$ be concentrated on $\graph D_k$. As $D_k = D$ $\mu^k-$almost everywhere, $\zeta^k$ is concentrated on $\graph D$. Therefore $\zeta \coloneqq m_k \zeta^k + \sum_{j=0, j \neq k}^d m_j 0_{\mu^j}$ is also concentrated on $\graph D$, and must belong to $\Sol_{\mu}^0$. By restriction, $\zeta^k = (\zeta \resmes \T A_k)/m_k$ belongs to $\Sol_{\mu^k}^0$. 
	The point \Cref{item:decomp:tan:con} is proved by repeating the argument with $D^{\perp}$ in place of $D$. 
\end{proof}

\section{Properties of the decomposition}

\filler{\Cref{res:decomp} only relies on the algebraic properties of tangent and solenoidal cones. We now give a refined description of the measures $\mu^k$ in terms of concentration on particular subsets. The argument is based on \Zaitchek's theorem, recalled below, which characterizes non-differentiability points of convex functions using \DC{k} sets. We start by some lemmas specific to \DC{k} and \sDC{k} sets, then return to the measures $(\mu^k)_k$.}

\renewcommand{\Sp}{\mathbb{R}^d}

\subsection{Preliminaries on \DC{k} sets}

\begin{definition}[\DC{k} and \sDC{k} sets]\label{def:sDCk}
	Let $k \in \llbracket0,d\rrbracket$.
	A set $A \subset \Sp$ is a \emph{difference of convex functions of dimension $k$}, denoted \DC{k}, if $A = \Phi(\mathbb{R}^k)$ for some function $\Phi : \mathbb{R}^k \to \mathbb{R}^d$ which, up to a permutation of coordinates, can be written as 
	\begin{align}\label{def:sDCk:statement}
		\Phi(X_k) = (X_k,(\varphi_{k+1}-\psi_{k+1})(X_k),\cdots,(\varphi_{d}-\psi_d)(X_k))
	\end{align}
	for $\varphi_j,\psi_j : \mathbb{R}^k \to \mathbb{R}$ convex functions and $j \in \llbracket k+1,d\rrbracket$. 
	A set that can be covered by countably many \DC{k} sets will be said \sDC{k}. 
\end{definition}

\DC{k} sets are called ``c--c hypersurfaces of dimension $k$'' in the original work of \Zaitchek \cite{zajicekDifferentiationConvexFunctions1979}, and ``$\delta-$convex surfaces of dimension $k$'' by \ncite{pavlicaPointsNondifferentiabilityConvex2004}. 
By convention, \DC{0} sets are points, and the only \DC{d} set is $\mathbb{R}^d$. 

\begin{rem}[Composition]\label{rem:compsDC}
	Let $0 \leqslant j \leqslant k \leqslant d$. 
	If $A \subset \mathbb{R}^k$ is \sDC{j}, and $\Phi : \mathbb{R}^k \to \mathbb{R}^d$ writes as a permutation of $(id_{\mathbb{R}^k},\varphi_{k+1}-\psi_{k+1},\cdots,\varphi_d-\psi_d)$ for convex functions $\varphi_j,\psi_j : \mathbb{R}^k \to \mathbb{R}$, then $\Phi(A)$ is \sDC{j} as well. Indeed, this is trivial if $j=0$ (since $A$ is countable) or $j=k$ (since $A = \mathbb{R}^k$). Otherwise, let $(A_n)_n$ be a countable family of \DC{j} sets covering $A$, each written as $\Phi_n(\mathbb{R}^j)$ for some $\Phi_n$ of the form \Cref{def:sDCk:statement}. The composition $\Phi \circ \Phi_n$ coincides with the identity on $j$ of its coordinates. Since compositions of DC functions are still DC by \cite[(I)~and~(II)]{hartmanFunctionsRepresentableDifference1959}, the remaining coordinates of $\Phi \circ \Phi_n$ are DC functions from $\mathbb{R}^j$ to $\mathbb{R}$. Therefore the sets $(\Phi \circ \Phi_n(A_n))_{n \in \mathbb{N}}$ are a countable family of \DC{j} sets covering $A$, i.e. $A$ is \sDC{j}.
\end{rem}

\begin{theorem*}[\Zaitchek's theorem]\label{res:zaitchek}
	Let $\varphi : \Sp \to \mathbb{R}$ be a convex function. Then each set 
	\begin{align}\label{def:Jk}
		J_k(\varphi) \coloneqq \left\{ x \in \Sp \st \dim \partial_x \varphi \geqslant d-k \right\}
	\end{align}
	is \sDC{k}, i.e. can be covered by countably many \DC{k} sets. Conversely, if $A \subset \Sp$ is \sDC{k}, there exists $\varphi : \Sp \to \mathbb{R}$ convex such that $A \subset J_k(\varphi)$. 
\end{theorem*}

\filler{If one is given a single \DC{k} set $A$, then the convex function can be chosen ``uniformly non-differentiable'' over $A$. This is merely an observation on the construction of \cite{zajicekDifferentiationConvexFunctions1979}, but will be used in the sequel.}

\begin{lem}\label{res:particularphi}
	Let $A = \Phi(\mathbb{R}^k) \subset \mathbb{R}^d$ be a \DC{k} set, with $\Phi$ as in \Cref{def:sDCk}. There exists a convex function $\phi : \mathbb{R}^d \to \mathbb{R}$ and $d-k+1$ measurable vector fields $f_k,\cdots,f_{d}$ such that $f_j(x) \in \partial_x \phi$ for each $x$, and whenever $x \in A$, the vectors $f_k(x),\cdots,f_d(x)$ are at distance at least one from each other.
\end{lem} 

\begin{proof}
	We may permute the coordinates of the space so as to write $\Phi = (id,\varphi_{k+1}-\psi_{k+1},\cdots,\varphi_{d}-\psi_d)$ for some convex functions $\varphi_j,\psi_j : \mathbb{R}^k \to \mathbb{R}$. Let $\phi \coloneqq \sum_{j = k+1}^d \phi_j$, where each $\phi_j$ is given by 
	\begin{align*}
		\phi_j(x_1,\cdots,x_d) \coloneqq \max\left(x_j + \psi_j(x_1,\cdots,x_k),\varphi_j(x_1,\cdots,x_k)\right).
	\end{align*}
	Then $\phi_j$ and $\phi$ are convex. For convenience, denote $x = (X^k,x_{k+1},\cdots,x_d)$ whenever $x \in \mathbb{R}^d$. For $j \in \llbracket k+1,d \rrbracket$, let $v_j,w_j : \mathbb{R}^{k} \to \mathbb{R}^{k}$ be measurable selections of $\partial \varphi_j$ and $\partial \psi_j$ respectively. Then $\partial_x \phi_j$ contains the vectors $(v_j(X^k),0) \in \mathbb{R}^d$ and $(w_j(X^k),0) + e_j$, with $e_j \in \mathbb{R}^d$ the $j^{\text{th}}$ vector of the canonical basis. Indeed, since $x_j = \varphi_j(X^k) - \psi_j(X^k)$, there holds for any $y = (Y^k,y_{k+1},\cdots,y_d)$ that
	\begin{align*}
		\phi_j(y)
		&= \max\left(y_j + \psi_j(Y^k), \varphi_j(Y^k)\right) \\
		&\geqslant \max\left(x_j + \psi_j(X^k) + (y_j-x_j) + \left<w_j(X^k),Y^k-X^k\right>, \varphi_j(X^k) + \left<v_j(X^k),Y^k-X^k\right>\right) \\
		&= \phi_j(x) + \max_{z \in \left\{(w_j(X^k),0)+e_j,(v_j(X^k),0)\right\}} \left<z, y-x\right>.
	\end{align*}
	Recall that $\partial_x \phi = \bigoplus_{j=k+1}^d \partial_x \phi_j$ \cite[10.9 and 7.27]{rockafellarConvexAnalysis1970}. Hence $\partial_x \phi$ contains  $\sum_{j=k+1}^d (v_j(X^k),0)$ and each $\sum_{j=k+1,j\neq \ell}^d (v_j(X^k),0) + (w_{\ell}(X^k),0) + e_{\ell}$ for $\ell \in \llbracket k+1,d\rrbracket$, which are at pairwise distance at least one. 
\end{proof}

\filler{The following lemma allows to pass from the direction of the sets of non-differentiability of a convex function to that of the subdifferential. A function $\varphi$ is semiconvex if $\varphi + \lambda |\cdot|^2/2$ is convex for some $\lambda \in \mathbb{R}$.}

\begin{lem}\label{res:affineortho}
	Let $\varphi : \mathbb{R}^d \to \mathbb{R}$ be semiconvex and $A \subset J_k(\varphi)$ be \DC{k}. Let $x \in A$ be such that
	\begin{itemize}
	\item the dimension of $\partial_x \varphi$ is exactly $d-k$, 
	\item the function $\Phi$ parametrizing $A$ as in \Cref{def:sDCk:statement} is differentiable at $x$,
	\item there exists $\varepsilon > 0$ and $r > 0$ sufficiently small so that for any $y \in A \cap \overline{\mathscr{B}}(x,r)$, the set $\partial_y \varphi$ contains $d-k+1$ vectors pairwise distant from each other by at least $\varepsilon$. 
	\end{itemize}
	Then $\text{Im} \nabla \Phi(x)$ is orthogonal to $\vspan \{\partial_x \varphi - p\}$ for any $p \in \partial_x \varphi$. 
\end{lem}

\filler{Note that the vector space $\vspan \{\partial_x \varphi - p\}$ is independent of the point $p \in \partial_x \varphi$. Intuitively, \Cref{res:affineortho} generalizes the observation that $\partial_x \varphi$ is orthogonal to any differentiable surface contained in its set of minimum points. The assumption that $\partial_y \varphi$ is sufficiently large for $y$ close to $x$ could be replaced, for instance by asking that all $\partial_y \varphi$ near $x$ contain a given (relative) interior point of $\partial_x \varphi$.}

\begin{proof}
	The function $\Phi : \mathbb{R}^k \to \mathbb{R}^d$ is injective, since it coincides with the identity on $k$ of its coordinates. Let $X \coloneqq \Phi^{-1}(x) \in \mathbb{R}^k$ and $(h_n)_{n} \subset (0,1)$ be a sequence converging to 0. Fix $e \in \mathbb{R}^k$, and denote $x_n \coloneqq \Phi(X + h_n e) \in A$. By assumption, for sufficiently large $n$, there exists vectors $(w_n^{\ell})_{\ell \in \llbracket k,d\rrbracket} \subset \partial_{x_n} \varphi$ at distance at least $\varepsilon$ from each other. Since $\varphi$ is locally Lipschitz, the sequences $(w_n^{\ell})_n$ are relatively compact. Extracting successively, we might assume that $w_n^{\ell} \to_n w^{\ell}$. By upper semicontinuity, each $w^{\ell}$ belongs to $\partial_x \varphi$, and is still at distance $\varepsilon$ from $w^{\ell'}$ for $\ell' \neq \ell$. Hence the vectors $(w^{\ell} - w^{k})_{\ell \in \llbracket k+1,d\rrbracket}$ span a space of dimension $d-k$, which must coincide with $\vspan \{\partial_x \varphi - w^k\}$ since the latter is of dimension $d-k$. 
	
	Denoting $\lambda$ the semiconvexity constant of $\varphi$, there holds for any $\ell,\ell' \in \llbracket k,d\rrbracket$ that
	\begin{align*}
		\varphi(x_n) 
		\geqslant \varphi(x) + \left<w^{\ell}, x_n - x\right> - \frac{\lambda}{2} |x_n-x|^2
		\geqslant \varphi(x_n) + \left<w_n^{\ell'}, x - x_n\right> + \left<w^{\ell}, x_n - x\right> - \lambda |x_n - x|^2.
	\end{align*}
	Dividing by $h_n > 0$ and sending $n \to \infty$, we get
	\begin{align*}
		0 
		\geqslant \lim_{n \to \infty} \left<w^{\ell} - w_n^{\ell'}, \frac{\Phi(X+h_n e) - \Phi(X)}{h_n}\right> - \frac{\lambda}{h_n} \left|\Phi(X + h_n e_j) - \Phi(X)\right|^2
		= \left<w^{\ell} - w^{\ell'}, \nabla \Phi(X)(e)\right>.
	\end{align*}
	Interchanging $\ell$ and $\ell'$, we get that $\text{Im} \nabla \Phi(X)$ is orthogonal to $\vspan \{\partial_x \varphi - w^k\}$, as claimed.
\end{proof}

\filler{We make use of the following quite strong definition of differentiability, tailored for \sDC{k} sets.}

\begin{definition}[Tangent plane to a \sDC{k} set]\label{def:tansDCk}
	Let $A$ be a \sDC{k} set contained in $\bigcup_{n} A_n$ for \DC{k} sets $A_n$. Let $\mathbb{N}_x \subset \mathbb{N}$ be the subset of $n \in \mathbb{N}$ such that $x \in A_n$, and for $n \in \mathbb{N}_x$, denote $\Phi_n$ a permutation of $(id_{\mathbb{R}^k},\varphi_{k+1}^n-\psi_{k+1}^n,\cdots,\varphi_d^n-\psi_d^n)$ such that $A_n = \Phi_n(\mathbb{R}^k)$. Then $A$ admits $P$ as a tangent plane at $x$ if for each $n \in \mathbb{N}_x$ and $j \in \llbracket k+1,d\rrbracket$, each $\varphi_j^n,\psi_j^n$ is differentiable at $\Phi_n^{-1}(x)$, and $\text{Im} \nabla \Phi_n(x) = P$.
\end{definition}

\begin{lem}\label{res:exiTansDCk}
	Let $(A_n)_{n \in \mathbb{N}}$ be a family of \DC{k} sets, and denote $A \coloneqq \bigcup_{n \in \mathbb{N}} A_n$. Then there exists a \sDC{k-1} set $B \subset A$ such that $A$ admits a tangent plane $P$ at any point $x \in A \setminus B$ in the sense of \Cref{def:tansDCk}.
\end{lem}

\begin{proof}
	For each $n$, denote $\Phi_n : \mathbb{R}^k \to \mathbb{R}$ a function as in \Cref{def:sDCk:statement} such that $A_n = \Phi_n(\mathbb{R}^k)$. Let $B^{(n)} \subset \mathbb{R}^k$ be the union of the sets of non-differentiability of the functions $\varphi_j,\psi_j$ for $j \in \llbracket k+1,d\rrbracket$. Each set $B^{(n)}$ is \sDC{k-1}, and by \Cref{rem:compsDC}, the composition $\Phi_n(B^{(n)}) \subset \mathbb{R}^d$ is still a \sDC{k-1} subset of $\mathbb{R}^d$. Denote $P_n(x)$ the image of $\nabla \Phi_n(x)$ at any point $x \in A_n \setminus B^{(n)}$. 
	
	We now construct a \sDC{k-1} subset of $A$ out of which the surfaces $A_n$ cannot intersect transversely. Precisely, for each $n \neq m$, let $B^{(n,m)}$ be the set of $x \in A_n \cap A_m \setminus (B^{(n)} \cup B^{(m)})$ such that $P_n(x) \neq P_m(x)$. We construct a convex function containing $B^{(n,m)}$ in its non-differentiability set $J_{k-1}$. By \Cref{res:particularphi}, there exist convex functions $\phi^{(n)},\phi^{(m)} : \mathbb{R}^d \to \mathbb{R}$ such that $A_n \subset J_k(\phi^{(n)})$, $A_m \subset J_k(\phi^{(m)})$, and the subdifferentials of $\phi^{(n)},\phi^{(m)}$ on $A_n,A_m$ contain uniformly separated points. Denote $B_{n,m} \coloneqq J_{k-1}(\phi^{(n)}) \cup J_{k-1} (\phi^{(m)})$, which is \sDC{k-1}. By \Cref{res:affineortho}, for any $x \in B^{(n,m)} \setminus B_{n,m}$, the affine space containing $\partial_x \phi^{(n)}$ is orthogonal to $P_n(x)$, and similarly for $\partial_x \phi^{(m)}$ and $P_m(x)$. Since $P_n(x) \neq P_m(x)$, we deduce that the affine hulls of $\partial_x \phi^{(n)}$ and $\partial_x \phi^{(m)}$ are not parallel. As a consequence, the convex set 
	\begin{align*}
		\partial_{x} (\phi^{(n)} + \phi^{(m)}) = \left\{ v+w \st v \in \partial_x \phi^{(n)}, \ w \in \partial_x \phi^{(m)} \right\}
	\end{align*}
	has dimension at least $d-k+1$. Hence $B^{(n,m)} \setminus B_{n,m}$ is contained in $J_{k-1}(\phi^{(n)} + \phi^{(m)})$, which is \sDC{k-1} since $\phi^{(n)} + \phi^{(m)}$ is convex. Adding back $B_{n,m}$, we obtain that $B^{(n,m)}$ is \sDC{k-1}. 
	
	Denote $B \coloneqq \bigcup_{n \in \mathbb{N}} B^{(n)} \cup \bigcup_{n,m \in \mathbb{N}^2} B^{(n,m)}$. Then $B$ is \sDC{k-1}, and if $x \in A \setminus B$, all sets $A_n$ containing $x$ admit a plane $P_n(x)$ as before. Additionally, $P_n(x) = P_m(x) \eqqcolon P(x)$ for any $n \neq m$ such that $x \in A_n \cap A_m$, otherwise $x$ would belong to $B^{(n,m)} \subset B$. 
\end{proof}

\renewcommand{\Sp}{\Omega}

\subsection{Characterization of $\mu^k$ by concentration on \DC{k} sets}\label{sec:concentration}

\filler{We come back to measures over $\Sp = \mathbb{R}^d$, with two intermediate results relating the dimension of the map $D^{\Tan} = D^{\perp}$ appearing in \Cref{res:decomp}, and the size of sets on which $\mu$ can be concentrated.}

\begin{prop}[Large $D^{\Tan}$ implies thin concentrations]\label{res:idenmuk:concentration}
	Let $0 \leqslant k \leqslant d-1$, and $\mu \in \Psp_2(\Sp)$ be a measure such that the Grassmannian section $D^{\Tan} : \Sp \rightrightarrows \mathbb{R}^d$ characterizing $\Tan_{\mu}^0$ has dimension larger than $d-k$ at $\mu-$almost every point. Then $\mu$ is concentrated on a \sDC{k} set. 
\end{prop}

\filler{Recall from \Cref{def:Jk} that for $\varphi : \Sp \to \mathbb{R}$ semiconvex, $J_k(\varphi) \coloneqq \left\{ x \in \Sp \st \dim \partial_x \varphi \geqslant d - k \right\}$.}

\begin{proof}
	We first consider the case of a compactly supported measure, then proceed by exhaustion. Precisely, in the two first steps, we show that whenever $\mu$ is compactly supported and such that $D^{\Tan} = D^{\Tan}_{\mu}$ has dimension $d-k$, then there exists a measurable \sDC{k} set $B$ such that $\mu(B) \geqslant 1/2$.

	\stephighlight{Compact case: construction of $B$}
	Let $(f^{\bdot}_j)_{j \in \llbracket k+1,d\rrbracket}$ be measurable selections of $D^{\Tan}$ such that $(f^{\bdot}_{k+1}(x),\cdots,f^{\bdot}_d(x))$ is an orthonormal basis of $D^{\Tan}(x)$ for any $x \in \Sp$. Such applications can be constructed by the Gram-Schmidt orthogonalization algorithm applied to a Castaing representation of $D^{\Tan}$, provided by \cite[Theorem 3.(d)]{rockafellarMeasurableDependenceConvex1969}. By \Cref{res:charSet0}, the measure field $\xi \coloneqq \frac{1}{d-k} \sum_{j=k+1}^d \frac{1}{2} \left[(id,-f_j^{\bdot})_{\#} \mu + (id,f_j^{\bdot})_{\#} \mu\right]$ belongs to $\Tan_{\mu}^0$. We show by a Chebyshev inequality that if $\eta$ is close to $\xi$ with respect to $W_{\mu}$, then it must put mass on balls around each $f_j^{\bdot}(x)$ for any $x$ belonging to a set of large $\mu-$measure. 
	
	Let $j \in \llbracket k+1,d\rrbracket$ and $s \in \{\pm 1\}$. Since $(x,v) \mapsto |v - s f_j^{\bdot}(x)|$ is measurable from $\T\Sp$ to $\mathbb{R}$, the sets
	\begin{align*}
		A_{j,s}
		\coloneqq \left\{ (x,v) \in \T\Sp \st |v - s f_j^{\bdot}(x)| \leqslant 1/2 \right\}
		\subset \T\Sp
	\end{align*}
	are measurable. Given a set $A \subset \T\Sp$, denote $A^x \coloneqq \left\{ v \st (x,v) \in A \right\}$. By the disintegration theorem of \cite[\nopp 10.4.15]{bogachevMeasureTheory2007}, for any $\eta \in \Psp_2(\T\Sp)_{\mu}$, there exists measures $(\eta_x)_{x \in \Sp}$ with $\eta_x \in \Psp(\mathbb{R}^d)$ such that for any $A \subset \T\Sp$ measurable, $x \mapsto \eta_x\left(A^x\right)$ is measurable, and $\eta(A) = \int_{x \in \Sp} \eta_x\left(A^x\right) d\mu(x)$. Therefore, the set 
	\begin{align*}
		C \coloneqq \left\{ x \in \Sp \st \exists (j,s) \in \llbracket k+1,d\rrbracket \times \{\pm 1\} \text{ such that } \eta_x(A_{j,s}^x) = \eta_x\left(\overline{\mathscr{B}}(s f_j^{\bdot}(x),1/2)\right) \leqslant \frac{1/2}{2(d-k)} \right\}
	\end{align*}
	is a measurable subset of $\Sp$. Let $\xi_x = \frac{1}{d-k} \sum_{j=k+1}^d \frac{1}{2} \left[\delta_{(x,-f_j^{\bdot}(x))} + \delta_{(x,f_j^{\bdot}(x))}\right]$ be a particular disintegration of $\xi$. Whenever $x \in C$, the measure $\xi_x$ places a mass $\frac{1}{2(d-k)}$ on each $\pm f_j^{\bdot}(x)$, whereas $\eta_x$ places less than $\frac{1/2}{2(d-k)}$ on at least one ball of radius $1/2$ centred in these points. Consequently, any transport plan $\alpha_x$ between $\eta_x$ and $\xi_x$ will force a mass of at least $\frac{1/2}{2(d-k)}$ to travel from a distance superior to $1/2$. In integral form, $\int_{v \in \mathbb{R}^d} |v|^2 d\alpha_x \geqslant \frac{1/2}{2(d-k)} (1/2)^2$ for any $x \in C$. Passing to the infimum over $\alpha_x$, this yields $W^2(\eta_x,\xi_x) \geqslant \frac{1}{16(d-k)}$, and integrating over $x$, we get that $\vphantom{\frac{1}{16(d-k)}}$ 
	\begin{align}\label{idenmuk:concentration:chebychev}
		\mu\left(C\right)
		= \int_{x \in \Sp} \ind_{C}(x) d\mu(x)
		\leqslant \int_{x \in \Sp} \frac{W^2(\eta_x,\xi_x)}{1/(16(d-k))} d\mu(x)
		\leqslant 16(d-k) W^2(\eta,\xi).
	\end{align}
	Here the last inequality stands since $W^2(\xi,\eta) = \int_{x \in \Sp} W^2(\xi_x,\eta_x) d\mu$ \cite[Prop.~4.2]{gigliGeometrySpaceProbability2008}. Since $\xi$ is tangent, there exists $\eta \in \Psp_2(\T\Sp)_{\mu}$ inducing a geodesic on $[0,\tau]$ for $\tau > 0$, and such that $16 (d-k) W^2(\eta,\xi) \leqslant \varepsilon$. Let $B \coloneqq C^{c}$. By \Cref{idenmuk:concentration:chebychev}, $\mu(B) \geqslant 1 - \varepsilon$. By definition of $C$, $\eta_x$ puts a mass of at least $\frac{1/2}{2(d-k)}$ on each $\overline{\mathscr{B}}(\pm f_j^{\bdot}(x),1/2)$ for $x \in B$, so that the restricted measure 
	\begin{align*}
		\gamma \coloneqq \eta \resmes \bigcup\displaystyle_{j \in \llbracket k+1,d\rrbracket}^{s \in \{\pm 1\}} A_{j,s}
	\end{align*}
	still puts mass on $\T B$, i.e. $B \subset \supp \pi_{x \#} \gamma$. Moreover, by restriction of optimality \cite[Theorem 4.6]{villaniOptimalTransport2009}, $\gamma$ induces a geodesic on $[0,\tau]$. 
	
	\stephighlight{Compact case: $B$ is \sDC{k}}
	Define $\varphi$ as the classical explicit Kantorovich potential \cite[(5.13)]{villaniOptimalTransport2009}, i.e.
	\begin{align*}
		\varphi(x) \coloneqq \frac{1}{2 \tau} \sup \left\{ |\tau v_N|^2 - |x_N + \tau v_N - x|^2 + \sum_{i=0}^{N-1}\left(|\tau v_i|^2 - |x_i + \tau v_i - x_{i+1}|^2\right) \st \begin{matrix}
		N \in \mathbb{N}, \text{ and} \\ (x_i,v_i)_{i=1}^N \subset \supp \gamma\end{matrix} \right\}.
	\end{align*}
	Then $\varphi(x) \geqslant (|\tau v_0|^2 - |x_0 + \tau v_0 - x|^2)/(2 \tau)$ by taking $N = 0$. In particular, $\varphi(x_0) \geqslant 0$, and $\varphi(x_0) \leqslant 0$ by cyclical monotonicity. Moreover, $\varphi(x) + |x|^2/(2\tau)$ is a supremum of convex functions, hence $\varphi$ is semiconvex. Using that $\sup_{a \in A} f(a) - \sup_{a \in A} g(a) \leqslant \sup_{a \in A} f(a)-g(a)$ whenever $\sup_{A} g < \infty$, there holds 
	\begin{align*}
		\varphi(x) - \varphi(x_0) 
		&\leqslant \frac{1}{2 \tau} \sup_{(x_N,v_N) \in \supp \gamma} |x_N + \tau v_N - x_0|^2 - |x_N + \tau v_N - x|^2 \\
		&= \frac{|x_0|^2 - |x|^2}{2 \tau} + \frac{|x - x_0|}{\tau} \sup_{(x_N,v_N) \in \supp \gamma} |x_N + \tau v_N|,
	\end{align*}
	with the last supremum bounded by $\diam \supp \mu + \tau (1+1/2) < \infty$. Therefore, $\varphi$ is real-valued. We now show that if $(x,v) \in \supp \gamma$, then $v \in \partial_x \varphi$. For any $\iota > 0$, let $(x_i,v_i)_{i=1}^{N} \subset \supp \gamma$ be $\iota-$optimal for the definition of $\varphi(x)$. Appending $(x,v)$ to the sequence $(x_i,v_i)_i$, we get that for any $y \in \Sp$,
	\begin{align*}
		\varphi(y) 
		&\geqslant \frac{1}{2 \tau} \left[|\tau v|^2 - |x + \tau v - y|^2 + |\tau v_N|^2 - |x_N + \tau v_N - x|^2 + \sum_{i=0}^{N-1} \left(|\tau v_i|^2 - |x_i + \tau v_i - x_{i+1}|^2\right)\right] \\
		&\geqslant \left<v, y - x\right> - \frac{|x - y|^2}{2 \tau} + \varphi(x) - \iota.
	\end{align*}
	Letting $\iota \searrow 0$, we obtain that $v \in \partial_x \varphi$. As $\gamma_x$ puts mass on balls of radius $1/2$ around each $\pm f_j^{\bdot}(x)$ for any $x \in B \subset \supp \mu$, with $|\pm f_j^{\bdot}(x)| = 1$, the convex set $\partial_x \varphi$ has dimension at least $d-k$, so that $B \subset J_k(\varphi)$. By \Zaitchek's theorem, $B$ is \sDC{k}.
	
	\stephighlight{General case: exhaustion}
	Consider now $\mu \in \Psp_2(\Sp)$ arbitrary. 
	Let $R_0 > 0$ be large enough such that $\mu(\overline{\mathscr{B}}(0,R_0)) \geqslant 1/2$. By the restriction formula of \Cref{res:restr:Tan}, the centred tangent cone to $\mu_{0} \coloneqq \mu \resmes \overline{\mathscr{B}}(0,R_0) / \mu(\overline{\mathscr{B}}(0,R_0))$ shares the application $D^{\Tan}$ with $\mu$. By the previous steps, there exists a \sDC{k} set $B_0 \subset \overline{\mathscr{B}}(0,R_0)$ such that $\mu_0(B_0) \geqslant 1/2$, which implies $\mu(B_0) \geqslant 1/4$. 
	
	Assume an increasing sequence of radii $R_j$ and measurable \sDC{k} sets $B_j$ has been constructed such that $\mu(B_j) \geqslant \mu(B_{j-1}) + (1-\mu(B_{j-1}))/4$ for $j \in \llbracket 1, n\rrbracket$. Let $R_{n+1}$ be such that $\mu(B_n^c \cap \mathscr{B}(0,R_{n+1})) \geqslant \mu(B_n^c)/2$, and apply the previous steps to the measure $\mu_{n+1} \coloneqq \mu \resmes (B_n^c \cap \mathscr{B}(0,R_{n+1})) / \mu(B_n^c \cap \mathscr{B}(0,R_{n+1}))$. This yields a measurable \sDC{k} set $B$ such that $\mu_{n+1}(B) \geqslant 1/2$. Up to restriction, we may consider that $B \subset \overline{\mathscr{B}}(0,R_{n+1})$ and is disjoint from $B_n$. Define $B_{n+1} \coloneqq B_n \cup B$. Then 
	\begin{align*}
		\mu(B_{n+1})
		= \mu(B_n) + \mu(B_n^c \cap \mathscr{B}(0,R_{n+1})) \times \mu_{n+1}(B)
		\geqslant \mu(B_n) + \frac{\mu(B_n^c)}{4}
		= \mu(B_n) + \frac{1 - \mu(B_n)}{4}.
	\end{align*} 
	The sequence $p_n \coloneqq \mu(B_n)$ satisfies $p_0 \geqslant \frac{1}{4}$ and $p_{n+1} \geqslant \frac{1+3 p_n}{4}$, so $\lim_{n \to \infty} p_n = 1$. 
	Therefore, the limit $B_{\infty}$ of the increasing sequence $(B_n)_n$ satisfies $\mu(B_{\infty}) = 1$. 
	As a countable union of \sDC{k} sets, $B_{\infty}$ is \sDC{k}.
\end{proof}

\begin{rem}[Support]
	The conclusion of \Cref{res:idenmuk:concentration} cannot be improved to covering the support of $\mu$ by \DC{k} sets. Indeed, consider $\mu = \sum_{n \in \mathbb{N}} \alpha_n \delta_{x_n}$ for a sequence $(x_n)_{n \in \mathbb{N}}$ dense in $[0,1]$. By \cite[Proposition 2.9]{aussedatStructureGeometricTangent2025}, $\Tan_{\mu}^0$ is equal to $\Psp_2(\T\Sp)_{\mu}^0$, so that $D^{\Tan}$ can be taken identically equal to $\mathbb{R}$. The measure $\mu$ is concentrated on the \sDC{0} (countable) set $\bigcup_{n \in \mathbb{N}} \{x_n\}$, but its support is $[0,1]$, which is not countable. 
\end{rem}

\filler{The following result is a partial converse of \Cref{res:idenmuk:concentration}.}

\begin{lem}[Thin concentration implies large $D^{\Tan}$]\label{res:idenmuk:lowerdim}
	Let $0 \leqslant k \leqslant d-1$, and $\mu \in \Psp_2(\Sp)$ put a positive mass on a measurable \DC{k} set $A \subset \Sp$. Then the application $D^{\Tan}$ characterizing $\Tan_{\mu}^0$ in \Cref{res:charSet0} satisfies $\dim D^{\Tan}(x) \geqslant d-k$ for $\mu-$almost every $x \in A$.
\end{lem}

\begin{proof}
	By \Cref{res:particularphi}, there exists a convex function $\varphi : \Sp \to \mathbb{R}$ and $d-k+1$ measurable vector fields $f_{k},\cdots,f_{d}$ such that $f_j(x) \in \partial_x \varphi$ for all $x$, and for any $x \in A$, $|f_j(x) - f_{\ell}(x)| \geqslant 1$ for $j \neq \ell \in \llbracket k,d\rrbracket$. Let 
	\begin{align*}
		\xi \coloneqq \frac{1}{d-k+1} \sum_{j=k+1}^d \frac{(id, f_j-id)_{\#} \mu + (id, f_k-id)_{\#} \mu}{2}.
	\end{align*}
	Then $(\pi_x,\pi_x+\pi_y)_{\#} \xi$ is concentrated on the graph of the convex function $\varphi$, and $\xi$ is the velocity of a geodesic. By \Cref{rem:barycentreddecomp}, the centred measure field $\xi^0 \coloneqq \left(\pi_x, \pi_v - \bary{\xi}{}(\pi_x)\right)_{\#} \xi$ is also tangent, and concentrated on $\big(x,\pm \frac{f_j(x) - f_k(x)}{2}\big)$ for $j \in \llbracket k+1, d\rrbracket$. Since the latter vectors are independent, $\xi^0$ splits mass in at least $d-k$ directions, so that $\dim D^{\Tan}(x) \geqslant d-k$ for $\mu-$almost every $x \in A$. 
\end{proof}

\filler{Combining both results, we arrive at the following statement.}

\begin{theorem}\label{res:charmuk}
	Let $\mu \in \Psp_2(\Sp)$ be decomposed as $\mu = \sum_{k=0}^d m_k \mu^k$ according to \Cref{res:decomp}. Then, for each $k \in \llbracket 0,d \rrbracket$ such that $m_k > 0$, the measure $\mu^k$ is concentrated on a \sDC{k} set and gives 0 mass to any \sDC{k-1} set. Moreover, for all $k \in \llbracket 0,d\rrbracket$,
	\begin{align}\label{charmuk:statement}
		m_k \, \mu^k(A) = \max_{\substack{B \text{ \sDC{k}} \\ \text{measurable }}} \min_{\substack{\vphantom{\int} C \text{ \sDC{k-1}} \\ \text{measurable }}} \mu(A \cap B \setminus C)
		\qquad \text{ for any measurable } A \subset \Sp,
	\end{align} 
	and the order of $\min$ and $\max$ can be reversed.
\end{theorem}

\begin{proof}
	For each $k$ such that $m_k > 0$, the Grassmannian section $D_k$ characterizing $\Tan_{\mu^k}^0$ has images of dimension $d-k$. Hence, whenever $k < d$, \Cref{res:idenmuk:concentration} implies that $\mu^k$ is concentrated on a \sDC{k} set. If $k = d$, the only \sDC{d} set is $\Sp$, and the statement is vacuous. For $k = 0$, we get that $\mu^0$ is countable. Moreover, if $k > 0$, then $\mu^k$ must give 0 mass to any \DC{k-1} set, otherwise \Cref{res:idenmuk:lowerdim} would imply that $\dim D_k \geqslant d-(k-1) = d-k+1$ on a set of positive $\mu^k-$measure. For each $k$ such that $m_k > 0$, let $A_k$ be a measurable \sDC{k} set on which $\mu^k$ is concentrated, and define $A_k = \emptyset$ if $m_k = 0$.
	
	Let $k \in \llbracket 0,d\rrbracket$, and $A,B,C \subset \mathbb{R}^d$ be measurable such that $B$ is \sDC{k} and $C$ is \sDC{k-1}. Since $\mu^{j}$ gives 0 mass to any \sDC{j-1} set, it gives 0 mass to any \sDC{$\ell$} set for $\ell \leqslant j-1$. Then
	\begin{align}\label{charmuk:sumformula}
		\mu(A \cap B \setminus C)
		= \sum_{j=0}^d m_j \mu^j(A \cap B \setminus C)
		= \sum_{j=0}^{k-1} m_j \mu^j(A \cap B \setminus C) + m_k \mu^k(A \cap B).
	\end{align}
	In particular, $\mu(A \cap B \setminus C) \geqslant m_k \mu^k(A \cap B)$ for any such $C$, with equality if $C$ is chosen as a \sDC{k-1} set containing $\bigcup_{j \in \llbracket 0,k-1\rrbracket} A_j$. We get that 
	\begin{align*}
		\min_{\substack{\vphantom{\int} C \text{ \sDC{k-1}} \\ \text{measurable }}} \mu(A \cap B \setminus C)
		= m_k \mu^k(A \cap B).
	\end{align*}
	Since $m_k \mu^k(A \cap B) \leqslant m_k \mu^k(A)$, with equality if $B = A_k$, we can take the maximum over \sDC{k} sets $B$ to obtain \Cref{charmuk:statement}. Using that $\mu^j(A \cap B \setminus C) \leqslant \mu^j(B \setminus C)$, we can also first take the maximum over $A$ in \Cref{charmuk:sumformula}, then the minimum over $C$, to obtain the same result.	
\end{proof}

\filler{As a consequence of the explicit formula \Cref{charmuk:statement}, the measures $\mu^k$ in \Cref{res:decomp} are uniquely defined (if $m_k > 0$). For $k = 0$, the measure $\mu^0$ is countable, and collects the atoms of $\mu$; for $k = d$, the measure $\mu^d$ gives 0 mass to any \DC{d-1} set, and is the transport-regular part of $\mu$. Let us detail some examples.
\begin{enumerate}[label={\emph{Example \Alph*.\ }},ref={\Alph*},leftmargin=0em,itemindent=5.9em]
\item\label{ex:dim1} In dimension one, the decomposition reduces to $\mu = m_0 \mu^0 + m_1 \mu^1$, where $\mu^0$ is supported on a \sDC{0} set and $\mu^1$ gives 0 mass to \sDC{0} sets. Since \sDC{0} sets are exactly countable sets, we get that $\mu^0$ is the (normalized) atomic part of $\mu$, while $\mu^1$ is its (normalized) diffuse part. Note that the Cantor part is classified in the same component as the absolutely continuous part of $\mu$. 

\item\label{ex:C2} If $\mu = \mathcal{H}^k \resmes \mathcal{M}$ for $\mathcal{H}^k$ the Hausdorff measure of dimension $k$, and $\mathcal{M}$ some $k-$dimensional $\mathcal{C}^2-$manifold with $\mathcal{H}^k(\mathcal{M}) = 1$, then $\mu = \mu^k$. Indeed, covering $\mathcal{M}$ by countably many  $\mathcal{C}^2$ charts, and using that $\mathcal{C}^2$ maps are DC \cite{hiriart-urrutyGeneralizedDifferentiabilityDuality1985}, one sees that $\mathcal{M}$ is \sDC{k}. Moreover, \DC{j} sets for $j < k$ are negligible for $\mathcal{H}^k$ (as non-differentiability sets of convex functions from $\mathbb{R}^k$ to $\mathbb{R}$). The fact that in this case, the centred elements of the geometric tangent cone are concentrated on the normal directions to $\mathcal{M}$ has been observed by Lott~\cite{lottTangentConesWasserstein2016}, and will be generalized to any $\mu \in \Psp_2(\Sp)$ in the next section. 

\item\label{ex:notC1} The regularity in the previous example cannot be weakened to $\mathcal{C}^1$. We sketch a counterexample in dimension $d=2$, taking inspiration from \cite{zajicekDifferentiationConvexFunctions1979,juilletDisplacementInterpolationMeasures2011}. Let $f \in \mathcal{C}([0,1];[0,1])$ be continuous and nowhere approximately differentiable, given for instance by \cite[Chap.~6]{kharazishviliStrangeFunctionsReal2006}. Define $F : x \mapsto \int_{y = 0}^x f(y) dy$ and $\mu \coloneqq (id,F)_{\#} \mathcal{L}_{[0,1]}$, with $\mathcal{L}$ the Lebesgue measure. We claim that $\mu = \mu^2$, i.e. $D^{\Sol} \equiv \mathbb{R}^2$. 

By the above characterization, it suffices to prove that $\mu$ gives 0 mass to any \DC{1} set. By \cite{anzellottiCkrectifiableSets1994}, any \DC{1} set can be covered by countably many $\mathcal{C}^2$ manifolds up to a set of $\mathcal{H}^1-$measure 0. As $\mu$ is absolutely continuous with respect to $\mathcal{H}^1$, it is enough to show that $\mu(A) = 0$ for any $\mathcal{C}^2$ manifold $A$. Using a parametrization of $A$ by graphs of $\mathcal{C}^2$ functions, this would be implied by the fact that $\mathcal{L}_{[0,1]}(\{F = \varphi\}) = 0$ for any $\varphi \in \mathcal{C}^2([0,1];\mathbb{R})$. Now, if $\{F = \varphi\}$ has positive Lebesgue measure, one shows that $f(x) = \varphi'(x)$ for any density point $x \in \{F = \varphi\}$. Therefore $\mathcal{L}(\{f = \varphi'\}) > 0$, and since $\varphi' \in \mathcal{C}^1$, the function $f$ admits $\varphi''$ as an approximate differential at any density point of $\{f = \varphi'\}$, contradicting the choice of $f$.
\end{enumerate}
}

\begin{rem}[Difference with the decomposability bundle]
	In \cite{albertiDifferentiabilityLipschitzFunctions2016}, the authors introduce the decomposability bundle of a measure $\mu$ as a Grassmannian section constructed (roughly) as follows: consider every possible representation of $\mu$ as a superposition of 1-Hausdorff measures supported on 1-rectifiable sets, and let $D^{\text{AM}}$ be the essential union of the approximate tangent planes to each piece. If $\mu = (id,F)_{\#} \mathcal{L}_{[0,1]}$ for $F \in \mathcal{C}^1([0,1];\mathbb{R})$, then $D^{\text{AM}}(x,F(x))$ reduces to the classical tangent plane $\mathbb{R} \cdot (1,F'(x))$. This differs from $D^{\Sol}$ in example~\Cref{ex:notC1}. Hence the following question: can $D^{\text{AM}}$ be linked to a Wasserstein tangent cone for another cost?
\end{rem}

\subsection{Relations with other notions of tangency}\label{sec:tangency}

\filler{The reader could rightfully complain that so far, no information has been provided concerning the direction of the Grassmannian section $D^{\Sol}$ in \Cref{res:charTanSol0}. We now correct this in two successive steps; first, if $\mu$ is concentrated on a \sDC{k} set $A$, we show that the tangent planes to $A$ are aligned with $D^{\Sol}$ at least $\mu-$almost everywhere. The definition of tangent plane involved is rather strong, and might be useful in applications; however, it uses information coming from a particular set over which $\mu$ is concentrated, and lacks an ``intrinsic'' character. Therefore, in a second step, we show that Preiss tangent measures are concentrated on $D^{\Sol}$ for $\mu-$almost every point.}

\filler{In this section, we consider only $k \in \llbracket 1,d-1\rrbracket$, since in the extremal cases $k = 0$ and $k = d$, the application $D^{\Sol}$ is either $\{0\}$ or the whole space $\mathbb{R}^d$}. 

\begin{prop}[$\mu^k$ aligns with $D^{\Sol}_k$]\label{res:PeqD}
	Let $k \in \llbracket 1,d-1\rrbracket$. Assume that $\mu \in \Psp_2(\Sp)$ is concentrated on a \sDC{k} set $A$ and gives 0 mass to any \sDC{k-1} set. Then for $\mu-$almost every point $x$, $A$ admits a tangent plane $P(x)$ in the sense of \Cref{def:tansDCk}, and $P(x) = D^{\Sol}(x)$. 
\end{prop}

\begin{proof}
	Let $(A_n)_n$ be a cover of $A$ by \DC{k} sets. By \Cref{res:exiTansDCk}, there exists a \sDC{k-1} set $B_0 \subset A$ such that $A$ admits a tangent plane $P = P(x)$ for any $x \in A \setminus B_0$, in the sense given to it in \Cref{def:tansDCk}. Our aim is to show that $P$ is orthogonal to $D^{\Tan}$ $\mu-$almost everywhere. 
	
	The beginning of the argument follows that of \Cref{res:idenmuk:lowerdim}. 
	For each $A_n$, let $\varphi_n : \Sp \to \mathbb{R}$ be a convex function given by \Cref{res:particularphi}, i.e. such that $A_n \subset J_k(\varphi_n)$ and $\partial \varphi_n$ admits $d-k+1$ measurable selections $f_k,\cdots,f_{d}$ that are at pairwise distance at least one over $A_n$. Then, for each $j \in \llbracket k+1,d\rrbracket$, the measure field $\frac{1}{2} (id,f_k-id)_{\#} \mu + \frac{1}{2} (id,f_j-id)_{\#} \mu$ induces a geodesic since $\varphi$ is convex, so that by \Cref{rem:barycentreddecomp}, its projection on centred measure fields
	\begin{align*}
		\frac{1}{2} \left(id,\frac{f_k - f_j}{2}\right)_{\hspace{-3pt}\#} \mu + \frac{1}{2} \left(id,\frac{f_j-f_k}{2}\right)_{\hspace{-3pt}\#} \mu
	\end{align*}
	is tangent. Since centred tangent measure fields are concentrated on $\graph D^{\Tan}$ by \Cref{res:decomp}, there exists a $\mu-$negligible set $B^{(n)}_0 \subset A$ such that $f_j(x) - f_k(x) \in D^{\Tan}(x)$ for any $x \in A \setminus B^{(n)}_0$. As $D^{\Tan}(x)$ is a subspace of dimension $d-k$, the independent vectors $f_j(x)-f_k(x)$ are spanning it. 
	
	We now use the particular choice of $\varphi_n$. Recall that for any $x \in A_n \setminus B_0$, the parametrization of $A_n$ is differentiable with differential spanning $P$. Denote $B^{(n)}_1 \coloneqq J_{k-1}(\varphi_n)$ the set of $x$ such that $\dim \partial_x \varphi_n > d-k$, which is \sDC{k-1}. By \Cref{res:affineortho}, for any $x \in A_n \setminus (B_0 \cup B^{(n)}_1)$, 
	\begin{align*}
		\vspan \left\{ f_j(x) - f_k(x) \st j \in \llbracket k+1,d\rrbracket \right\} = P(x)^{\perp}.
	\end{align*}
	Hence, for any $x \in A_n \setminus \big(B_0 \cup B^{(n)}_0 \cup B^{(n)}_1\big)$, there holds $D^{\Tan}(x) = P(x)^{\perp}$. Since $\mu$ gives 0 mass to \sDC{k-1} sets, both $B_0$ and $B^{(n)}_1$ are $\mu-$negligible. Hence $B \coloneqq \bigcup_{n \in \mathbb{N}} B_0 \cup  B^{(n)}_0 \cup B^{(n)}_1$ is a $\mu-$negligible set such that $P(x) = D^{\Tan}(x)^{\perp} = D^{\Sol}(x)$ for any $x \in A \setminus B$, as claimed.
\end{proof}

\filler{The following definition is extracted from \cite{preissGeometryMeasuresDistribution1987}, where it is shown that any measure $\mu$ admits at least one tangent measure $\nu$ at $\mu-$almost any point.}

\begin{definition}[Tangent measure in the sense of Preiss]\label{def:TanPreiss}
	Let $\mu \in \Psp(\Sp)$ and $x \in \supp \mu$. A locally finite measure $\nu \in \M_+(\Sp)$ is tangent to $\mu$ at $x$ in the sense of Preiss if there exists a vanishing sequence $(h_n)_{n \in \mathbb{N}} \subset (0,1]$ and $c > 0$ such that for any $\varphi \in \mathcal{C}_c(\Sp;\mathbb{R})$, 
	\begin{align*}
		\int_{x \in \Sp} \varphi(x) d\nu(x)
		= \lim_{n \to \infty} \frac{c}{\mu\left(\mathscr{B}(x,h_n)\right)} \int_{y \in \Sp} \varphi\left(\frac{y-x}{h_n}\right) d\mu(y).
	\end{align*}
\end{definition}

\begin{prop}\label{res:PreissT}
	Let $k \in \llbracket 1,d-1\rrbracket$. Assume that $\mu \in \Psp_2(\Sp)$ is concentrated on a \sDC{k} set $A$ and gives 0 mass to any \sDC{k-1} set. Then, for $\mu-$almost every point $x$, any tangent measure $\nu \in \M_+$ to $\mu$ at $x$ in the sense of \Cref{def:TanPreiss} is supported on $D^{\Sol}(x)$.
\end{prop}

\begin{proof}
	Let again $(A_n)_n$ be a cover of $A$ by \DC{k} sets. By \Cref{res:PeqD}, there exists a $\mu-$negligible set $B_0 \subset A$ such that for all $x \notin B_0$, $A$ admits $D^{\Sol}(x)$ as a tangent plane in the sense of \Cref{def:tansDCk}. For each $n$, let $A^{\cup n} \coloneqq \bigcup_{m \leqslant n} A_m$. Then $A = \bigcup_{n \in \mathbb{N}} A^{\cup n}$. By the density theorem \cite[\nopp 2.9.13]{federerGeometricMeasureTheory1996}, there exists a $\mu-$null set $B^{(n)} \subset A^{\cup n}$ such that every $x \in A^{\cup n} \setminus B^{(n)}$ is a density point of $A^{\cup n}$, i.e. 
	\begin{align*}
		\lim_{h \searrow 0} \frac{\mu\left((A^{\cup n})^c \cap \mathscr{B}(x,h)\right)}{\mu(\mathscr{B}(x,h))} = 0.
	\end{align*}
	Let $B_1 \coloneqq \bigcup_{n \in \mathbb{N}} B^{(n)}$. Then $B_1$ is $\mu-$negligible, and for any $x \in A \setminus B_1$, there exists $n_x \in \mathbb{N}$ such that $x$ is a density point of $A^{\cup n_x}$. 
	
	Let $x \in A \setminus (B_0 \cup B_1)$ and $\nu \in \M_+$ be a Preiss tangent measure at $x$, i.e. the limit in $C_c'$ of $\frac{c}{\mu\left(\mathscr{B}(x,h_n)\right)} (\frac{id - x}{h_n})_{\#} \mu$ for some $c > 0$ and vanishing sequence $(h_n)_{n \in \mathbb{N}} \subset (0,1]$. Denote $P = D^{\Sol}(x)$, and for any $\varepsilon > 0$, let $P^{\varepsilon} \coloneqq \left\{ y \in \Sp \st d(y,P) \leqslant \varepsilon \right\}$. To show that $\supp \nu \subset P$, it is enough to prove that for any $0 < \varepsilon < R$, there holds $\nu\left(\mathscr{B}(0,R) \setminus P^{\varepsilon}\right) = 0$. Let then $0 < \varepsilon < R$. We first note that 
	\begin{align}\label{PreissT:boundquot}
		\limsup_{n \to \infty} c\, \frac{\mu\left(\mathscr{B}(x,R h_n)\right)}{\mu\left(\mathscr{B}(x,h_n)\right)} < \infty.
	\end{align}
	Indeed, let $\psi \in \mathcal{C}(\Sp;[0,1])$ be supported in $\mathscr{B}(0,2R)$ and identically equal to one over $\overline{\mathscr{B}}(0,R)$. Then 
	\begin{align*}
		c \, \frac{\mu\left(\mathscr{B}(x,R h_n)\right)}{\mu\left(\mathscr{B}(x,h_n)\right)}
		\leqslant \frac{c}{\mu(\mathscr{B}(x,h_n))} \int_{y \in \Sp} \psi\left(\frac{y-x}{h_n}\right) d\mu(x)
		\underset{n \to \infty}{\longrightarrow} \int \psi d\nu 
		\leqslant \nu(\mathscr{B}(0,2R))
		< \infty.
	\end{align*}
	Let now $\varphi \in \mathcal{C}(\Sp;[0,1])$ be supported in $\mathscr{B}(0,R) \setminus P^{\varepsilon}$. By definition, 
	\begin{align}\label{PreissT:errors}
		\int \varphi d\nu
		&= \lim_{n \to \infty} \frac{c}{\mu\left(\mathscr{B}(x,h_n)\right)} \int_{y \in \Sp} \varphi\left(\frac{y - x}{h_n}\right) d\mu(y) 
		\leqslant \liminf_{n \to \infty} c \, \frac{\mu\left(\mathscr{B}(x, R h_n) \setminus (x+P)^{\varepsilon h_n}\right)}{\mu\left(\mathscr{B}(x,h_n)\right)} \notag \\
		&\leqslant \liminf_{n \to \infty} c \, \frac{\mu\left(A^{\cup n_x} \cap \mathscr{B}(x, R h_n) \setminus (x+P)^{\varepsilon h_n}\right)}{\mu\left(\mathscr{B}(x,h_n)\right)} + c \, \frac{\mu\left((A^{\cup n_x})^c \cap \mathscr{B}(x, R h_n)\right)}{\mu\left(\mathscr{B}(x,h_n)\right)}.
	\end{align}
	On the one hand, as $x \notin B_0$, each \DC{k} sets $A_{m}$ for $m \in \llbracket 0, n_x \rrbracket$ admits $P$ as a tangent plane in the sense of \Cref{def:tansDCk}, which implies that the parametrization $A_m = \Phi_m(\mathbb{R}^k)$ is differentiable at $x$ and satisfies $P = \text{Im} \nabla \Phi_m(x)$. Hence, for each $m \in \llbracket 0,n_x\rrbracket$, there exists $N_{x,m}$ large enough so that 
	\begin{align*}
		A_{m} \cap \mathscr{B}(x, R h_n) \subset (x+P)^{\varepsilon h_n}
		\qquad \forall n \geqslant N_{x,m}.
	\end{align*}
	For $N \geqslant \max_{m \in \llbracket 0,n_x\rrbracket} N_{x,m}$, the first term in \Cref{PreissT:errors} vanishes. On the other hand, 
	\begin{align*}
		c\, \frac{\mu\left((A^{\cup n_x})^c \cap \mathscr{B}(x, R h_n)\right)}{\mu\left(\mathscr{B}(x,h_n)\right)}
		= \frac{\mu\left((A^{\cup n_x})^c \cap \mathscr{B}(x, R h_n)\right)}{\mu\left(\mathscr{B}(x,R h_n)\right)} \times c\, \frac{\mu\left(\mathscr{B}(x,R h_n)\right)}{\mu\left(\mathscr{B}(x,h_n)\right)}.
	\end{align*}
	Since $x$ is a density point of $A^{\cup n_x}$, the first term of the product goes to 0 when $n \to \infty$. The second term being bounded by \Cref{PreissT:boundquot}, both terms in \Cref{PreissT:errors} are converging to 0 when $n \to \infty$, and $\int \varphi d\nu = 0$. This being valid for any $\varphi \in \mathcal{C}(\Sp;[0,1])$ supported in $\overline{\mathscr{B}}(0,R) \setminus P^{\varepsilon}$, we conclude that $\nu\left(\mathscr{B}(0,R) \setminus P^{\varepsilon}\right) = 0$. 
\end{proof}

\filler{It may happen that the inclusion $\supp \nu \subset D^{\Sol}(x)$ is strict. For instance, in example~\Cref{ex:notC1} at the end of \Cref{sec:concentration}, the measure $(id,F)_{\#} \mu$ with $F \in \mathcal{C}^1$ admits $(id,F'(x) \cdot id)_{\#} \mathcal{L}_{\mathbb{R}}$ as a unique Preiss tangent measure at $(x,F(x)) \in \mathbb{R}^2$. In particular, $\nu$ is supported on a line, but $D^{\Sol}$ is equal to $\mathbb{R}^2$ for $\mu-$almost every point.}

\paragraph{Acknowledgments}

The author would like to thank Adolfo Arroyo-Rabasa for relevant literature pointing. 

\section*{Appendix}\label[appx]{sec:appendix}

\begin{proof}[Proof of \Cref{res:projConvCone}]
	The existence and uniqueness of $\pi_C^{\mu} \xi$ follow verbatim the steps of \cite[Proposition 4.30]{gigliGeometrySpaceProbability2008}, mimicking the Hilbertian case. We show that for any $\alpha \in \Gamma_{\mu,o}(\xi,\pi_C^{\mu} \xi)$, the element $\gamma \coloneqq (\pi_x,\pi_v-\pi_w)_{\#} \alpha$ belongs to $C^{\perp}$. Let $\zeta \in C$ be arbitrary, and $\beta = \beta(\pi_x,\pi_u,\pi_v,\pi_w) \in \Gamma_{\mu}(\xi,\pi_C^{\mu} \xi,\zeta)$ with $(\pi_x,\pi_u,\pi_v)_{\#} \beta = \alpha$. By horizontal convexity, $(\pi_x,(1-\varepsilon) \pi_v + \varepsilon \pi_w)_{\#} \beta \in C$. Then
	\begin{align*}
		W_{\mu}^2(\xi,\pi_C^{\mu} \xi)
		&\leqslant W_{\mu}^2\left(\xi,(\pi_x,(1-\varepsilon) \pi_v + \varepsilon \pi_w)_{\#} \beta\right) 
		\leqslant \int_{(x,v,w) \in \T^2\Sp} \left|u - (1-\varepsilon) v - \varepsilon w\right|^2 d\beta \\
		&= W_{\mu}^2(\xi,\pi_C^{\mu} \xi) + 2 \varepsilon \int \left<u-v,v-w\right> d\beta + \varepsilon^2 \int |v-w|^2 d\beta.
	\end{align*}
	Taking $\varepsilon$ to 0, we get $0 \leqslant \int \left<u-v,v-w\right> d\beta$. Choosing in particular $\zeta = r \cdot \pi_C^{\mu} \xi \in C$ for $r > 0$, and $\beta = (\pi_x,\pi_v,\pi_w,r \pi_w)_{\#}\alpha$, we get that $\int \left<u-v,v\right> d\alpha = 0$ for any $\alpha \in \Gamma_{\mu,o}(\xi,\pi_C^{\mu} \xi)$. Therefore, in the general case of $\zeta \in C$,
	\begin{align*}
		\left<\gamma, \zeta\right>_{\mu}
		= \sup_{\substack{\beta \in \Gamma_{\mu}(\xi,\pi_C^{\mu} \xi,\zeta) \\ (\pi_x,\pi_u,\pi_v)_{\#} \beta = \alpha}} \int_{(x,u,v,w)} \left<u - v, w\right> d \beta
		\leqslant \int_{(x,u,v)} \left<u-v,v\right> d\alpha
		= 0.
	\end{align*}
	As the opposite inequality holds by \Cref{res:centredislocal}, we deduce that $\gamma \in C^{\perp}$.
	
	Now, for any $\zeta \in C^{\perp}$, we introduce artificially $\pi_C^{\mu} \xi$ in the definition of $W_{\mu}^2(\xi,\zeta)$ to get
	\begin{align*}
		W_{\mu}^2(\xi,\zeta) 
		&= \inf_{\substack{\beta \in \Gamma_{\mu}(\xi,\pi_C^{\mu} \xi, \zeta) \\ (\pi_x,\pi_u,\pi_v)_{\#} \beta = \alpha}} \int_{(x,u,v,w)} \left|v + (u - v - w)\right|^2 d\beta \\
		&= \|\pi_C^{\mu} \xi\|_{\mu}^2 + \inf_{\substack{\beta \in \Gamma_{\mu}(\xi,\pi_C^{\mu} \xi, \zeta) \\ (\pi_x,\pi_u,\pi_v)_{\#} \beta = \alpha}} 2 \int_{(x,u,v,w)} \left<v, u - v - w\right> d\beta + |(u-v) - w|^2 d\beta \\
		&\geqslant W_{\mu}^2(\xi,\gamma) + \inf_{\substack{\beta \in \Gamma_{\mu}(\xi,\pi_C^{\mu} \xi, \zeta) \\ (\pi_x,\pi_u,\pi_v)_{\#} \beta = \alpha}} 2 \int_{(x,u,v,w)} \left<v, - w\right> d\beta + W_{\mu}^2(\gamma, \zeta).
	\end{align*}
	Here we used respectively that $(\pi_x,\pi_v,\pi_v - \pi_w)_{\#} \alpha$ is a transport plan between $\xi$ and $\gamma$, that $\int \left<v,u-v\right> d\alpha = 0$, and that $(\pi_x,\pi_u-\pi_v,\pi_w)_{\#} \beta$ is a transport plan between $\gamma$ and $\zeta$. The middle term is exactly $- 2 \left<\pi_C^{\mu} \xi, \zeta\right>_{\mu}$, which vanishes since $\pi_C^{\mu} \xi \in C$ and $\zeta \in C^{\perp}$. Therefore, for any $\zeta \in C^{\perp}$, there holds $W_{\mu}^2(\xi,\zeta) \geqslant W_{\mu}^2(\xi,\gamma) + W_{\mu}^2(\gamma,\zeta)$, and $\gamma$ must be the metric projection of $\xi$ over $C^{\perp}$. 
	
	If now $\zeta \in C$, we have for any $\beta \in \Gamma_{\mu}(\xi,\pi_C^{\mu} \xi, \zeta)$ such that $(\pi_x,\pi_u,\pi_v)_{\#} \beta = \alpha$ that $\int \left<u - v, w\right> d\beta \leqslant \left<\gamma, \zeta\right>_{\mu} = 0$. We deduce that 
	\begin{align*}
		\left<\xi, \zeta\right>_{\mu}
		= \sup_{\substack{\beta \in \Gamma_{\mu}(\xi,\pi_C^{\mu} \xi, \zeta) \\ (\pi_x,\pi_u,\pi_v)_{\#} \beta = \alpha}} \int_{(x,u,v,w)} \left<u - v + v, w\right> d\beta
		\leqslant 0 + \sup_{\substack{\beta \in \Gamma_{\mu}(\xi,\pi_C^{\mu} \xi, \zeta) \\ (\pi_x,\pi_u,\pi_v)_{\#} \beta = \alpha}} \int_{(x,u,v,w)} \left<v, w\right> d\beta = \left<\pi_C^{\mu} \xi, \zeta\right>_{\mu}.
	\end{align*}
	If, in addition, $\zeta = -1 \cdot \zeta$, then 
	\begin{align*}
		\int \left<u - v, w\right> d\beta
		= - \int \left<u - v, - w\right> d\beta
		\geqslant - \left<\gamma, - 1 \cdot \zeta\right>_{\mu}
		= - \left<\gamma, \zeta\right>_{\mu}
		= 0,
	\end{align*}
	and equality holds in the penultimate line, so that $\left<\xi, \zeta\right>_{\mu} = \left<\pi_C^{\mu} \xi, \zeta\right>_{\mu}$.
\end{proof}

\begin{proof}[Proof of \Cref{res:extensionOpt}]
	Let $M \geqslant 0$ be such that $\supp \nu \subset \overline{\mathscr{B}}(0,M)$, and $(x_0,v_0) \in \supp (\pi_x,\pi_x+\pi_v)_{\#} \eta$. The formula
	\begin{align*}
		\varphi(x) &\coloneqq \sup \left\{ \sum_{i=0}^n |x_i - y_i|^2 - \sum_{i=0}^{n-1} |x_{i+1} - y_{i}|^2 - |x-y_n|^2 \st n \in \mathbb{N}, \ (x_i,y_i)_{i=1}^n \subset \supp (\pi_x,\pi_x+\pi_v)_{\#} \eta \right\}
	\end{align*}
	defines a semiconvex function from $\Sp$ to $\mathbb{R} \cup \{\infty\}$ with the property that $\varphi^c(y) - \varphi(x) = |x-y|^2$ for $(\pi_x,\pi_x+\pi_v)_{\#} \eta-$almost $(x,y)$ \cite[Theorem 5.10]{villaniOptimalTransport2009}. The support of $(\pi_x,\pi_x+\pi_v)_{\#} \eta$ is cyclically monotone, so taking $x = x_0$, we get that $\varphi(x_0) \leqslant 0$.
	Since each $y_i$ appearing in the supremum is contained in $\supp (\pi_x+\pi_v)_{\#} \eta = \supp \nu \subset \overline{\mathscr{B}}(0,M)$, one has
	\begin{align*}
		\varphi(x) - \varphi(x_0)
		&\leqslant \sup \left\{ -|x-y_n|^2 + |x_0-y_n|^2 \st \exists x_n \in \Sp \text{ with } (x_n,y_n) \in \supp (\pi_x,\pi_x+\pi_v)_{\#} \eta \right\} \\
		&\leqslant |x_0|^2 + 2 M |x_0 - x| - |x|^2.
	\end{align*}
	The function $\varphi$ is lower bounded by $-|x-y_0|^2 + |y_0-x_0|^2$ by definition, so locally bounded, hence locally Lipschitz since it is semiconvex. Therefore the set-valued subdifferential application $x \mapsto \partial_x \left(\varphi-\frac{|\cdot|^2}{2}\right)$ is compact-valued, and upper semicontinuous in the set-valued sense by \cite[Corollary 24.5.1]{rockafellarConvexAnalysis1970}. By the selection theorem \cite[\nopp 18.13]{aliprantisInfiniteDimensionalAnalysis2006}, it admits a measurable selection $f : \Sp \to \T\Sp$. Define then $\xi \coloneqq (1-\lambda) \eta + \lambda (f_{\#} \mu_2)$. By construction, $\xi$ is still supported on the subdifferential of $\varphi - |\cdot|^2/2$, so $(\pi_x,\pi_x+\pi_v)_{\#} \xi$ is cyclically monotone, hence optimal.
\end{proof}

\begin{proof}[Proof of \Cref{res:SolIfOrthCompSupp}]
	By definition, the set of measure fields of the form $\lambda \cdot (\pi_x,\pi_y-\pi_x)_{\#} \gamma$, where $\lambda \geqslant 0$ and $\gamma$ is optimal, is dense in $\Tan_{\mu}$ with respect to $W_{\mu}$. The metric scalar product being continuous and positively homogeneous, one has that $\zeta \in \Sol_{\mu}$ if and only if $\left<\zeta, (\pi_x,\pi_y-\pi_x)_{\#} \gamma\right>_{\mu} = 0$ for any $\gamma \in \Gamma_o(\mu,\nu)$ and $\nu \in \Psp_2(\Sp)$. To prove the claim, we have to show that it suffices to consider compactly supported $\nu$. Since we need convergence with respect to $W_{\mu}$, we cannot approximate the target measure by any compactly supported measure and use stability of optimality; hence we construct an explicit approximation. The results on c-transforms that we use may be found in \cite[Section 5]{villaniOptimalTransport2009}, in particular Theorem 5.10.

	Let $\eta \in \Gamma_o(\mu,\nu)$ for some $\nu \in \Psp_2(\Sp)$. Let $\varphi : \Sp \to \mathbb{R} \cup \{\infty\}$ be a proper c-convex Kantorovich potential for the pair $(\mu,\nu)$, and $\varphi^c : y \mapsto \inf_{x \in \Sp} \varphi(x) + |x-y|^2$ its c-transform, which satisfies $\varphi(x) = \sup_{y \in \Sp} \varphi^c(y) - |x-y|^2$. Denote $\Gamma \subset \Sp^2$ the set of $(x,y)$ such that $\varphi^c(y) - \varphi(x) = |x-y|^2$, which contains $\supp \eta$ and is cyclically monotone. Our strategy is to ``truncate'' $\Gamma$ on the $y$ variable, as follows.
	
	For each $R > 0$, define
	\begin{align*}
		\varphi_R(x) 
		\coloneqq \sup_{y \in \overline{\mathscr{B}}(0,R)} \varphi^c(y) - |x - y|^2
		= \sup_{y \in \Sp} \varphi^c(y) - |x - y|^2 - \chi_{R}(y),
	\end{align*}
	where $\chi_R(y) = 0$ if $|y| \leqslant R$, and $+\infty$ otherwise. 
	The function $\varphi_R$ is c-convex by definition, inferior to $\varphi$, lower bounded by a quadratic polynomial, and for any $x_0 \in \dom \varphi$,
	\begin{align*}
		\varphi_R(x) - \varphi_R(x_0)
		\leqslant \sup_{y \in \overline{\mathscr{B}}(0,R)} |x_0 - y|^2 - |x - y|^2
		\leqslant |x_0|^2 + 2 R |x_0 - x| - |x|^2.
	\end{align*}
	Hence $\varphi_R$ is locally bounded, and locally Lipschitz since semiconvex. 
	
	By a classical computation, the set $\Gamma_R \subset \Sp^2$ of pairs $(x,y)$ such that $\varphi_R(x) = \varphi^c(y) - |x-y|^2 - \chi_{R}(y)$ is cyclically monotone, and contained in $\Sp \times \overline{\mathscr{B}}(0,R)$. Note that $\Gamma \cap (\Sp \times \overline{\mathscr{B}}(0,R)) \subset \Gamma_R$; indeed, whenever $(x,y) \in \Gamma$ with $|y| \leqslant R$, then $\varphi(x) = \varphi^c(y) - |x-y|^2 \geqslant \varphi^c(z) - |x-z|^2$ for any $z \in \mathbb{R}^d$. In particular, the supremum on the ball of radius $R$ is also attained at $y$, and $\varphi_R(x) = \varphi(x) = \varphi^c(y) - |x-y|^2 - \chi_{R}(y)$. We deduce that 
	\begin{align}\label{SolIfOrthCompSupp:suppeta}
		\supp \eta \cap \left(\Sp \times \overline{\mathscr{B}}(0,R)\right) \subset \Gamma_R.
	\end{align}
	
	For each $R > 0$, the correspondence $x \mapsto \left\{ y \st (x,y) \in \Gamma_R \right\}$ is upper semicontinuous with compact images. By \cite[\nopp 18.13]{aliprantisInfiniteDimensionalAnalysis2006}, it admits a measurable selection $f_R : \Sp \to \Sp$, that satisfies $|f_R(x)| \leqslant R$ for all $x \in \Sp$. Define $\eta_R$ by replacing the part of $\eta$ that goes out of $\Sp \times \overline{\mathscr{B}}(0,R)$ by the measurable selection $f_R$; explicitly, 
	\begin{align*}
		\eta_R \coloneqq \eta \resmes \left(\Sp \times \overline{\mathscr{B}}(0,R)\right) + (\pi_x, f_R(\pi_x))_{\#} \eta \resmes (\Sp \times \overline{\mathscr{B}}(0,R)^c).
	\end{align*}
	Recalling \Cref{SolIfOrthCompSupp:suppeta}, $\eta_R$ is a probability measure concentrated on the cyclically monotone set $\Gamma_R$, hence an optimal transport plan between its marginals. The measure $\pi_{y \#} \eta_R$ is supported on $\overline{\mathscr{B}}(0,R)$ by construction, and since $|f_R(x)| \leqslant R$, one has
	\begin{align*}
		W_{\mu}^2\left((\pi_x,\pi_y-\pi_x)_{\#} \eta, (\pi_x,\pi_y-\pi_x)_{\#} \eta_R\right)
		\leqslant \int_{x,y \in \Sp, |y| > R} \left|f_R(x)-y\right|^2 d\eta
		\leqslant \int_{y \in \Sp, |y| > R} (R + |y|)^2 d\nu
		\underset{R \to \infty}{\longrightarrow} 0.
	\end{align*}
	As $\left<\zeta, (\pi_x,\pi_y-\pi_x)_{\#} \eta_R\right>_{\mu} = 0$ for any $R$ by assumption, we get that $\left<\zeta, (\pi_x,\pi_y-\pi_x)_{\#} \eta\right>_{\mu} = 0$. Since $\eta$ and $\nu$ were arbitrary, $\zeta$ is solenoidal.
\end{proof}

\paragraph{Conflict of interest} This work benefited from the support of the grant ANR-22-CE40-0010 while the author was in Rouen, and from the ERC starting grant ConFine n°101078057 in Pisa. For the purpose of open access, the author has applied a CC-BY public copyright licence to any Author Accepted Manuscript (AAM) version arising from this submission. Data availability is not applicable to this article as no new data were created or analysed in this study. 

\end{document}